\documentclass[11pt]{amsart}
\usepackage{hyperref,graphicx,dsfont,color,bbm}
\usepackage[latin1]{inputenc}
\usepackage[T1]{fontenc}
\usepackage[american]{babel}
\usepackage[a4paper,margin=3cm]{geometry}

\newcommand{\ABS}[1]{{\left| #1 \right|}} 
\newcommand{\PAR}[1]{{\left(#1\right)}} 
\newcommand{\BRA}[1]{{\left\{#1\right\}}} 

\newcommand{\ind}{\mathbbm{1}}

\newcommand{\dE}{\mathbb{E}}

\newcommand{\dP}{\mathbb{P}}

\newcommand{\dR}{\mathbb{R}}
\newcommand{\dV}{\mathbb{V}}

\newcommand{\cL}{{\mathcal L }}

\newcommand{\cN}{{\mathcal N }}

\newtheorem{thm}{Theorem}[section]
\newtheorem{cor}[thm]{Corollary}
\newtheorem{prop}[thm]{Proposition}
\newtheorem{lem}[thm]{Lemma}
\newtheorem{as}[thm]{Assumption}
\newtheorem{defi}[thm]{Definition}
\newtheorem{rem}[thm]{Remark}
\newtheorem{ex}[thm]{Example}

\title{On the length of one-dimensional reactive paths}

\author{Frédéric Cérou} %
\address[F. Cérou]{INRIA Rennes - Bretagne Atlantique, Campus 
de Beaulieu, 35042 Rennes Cedex, France}
\email{Frederic.Cerou@inria.fr}
\urladdr{http://www.irisa.fr/aspi/fcerou/}

\author{Arnaud Guyader} %
\address[A. Guyader]{INRIA Rennes - Bretagne Atlantique and 
Université de Haute Bretagne, Place du Recteur H. Le Moal, 
CS 24307, 35043 Rennes Cedex, France}
\email{arnaud.guyader@uhb.fr}
\urladdr{http://www.sites.univ-rennes2.fr/laboratoire-statistique/AGUYADER/}

\author{Tony Lelièvre} %
\address[T. Lelièvre]{CERMICS, École des Ponts ParisTech, 
6-8 avenue Blaise Pascal, 77455 Marne La Vallée, France}
\email{tony.lelievre@cermics.enpc.fr}
\urladdr{http://cermics.enpc.fr/~lelievre/}

\author{Florent~Malrieu}
\address[F. Malrieu]{IRMAR UMR CNRS 6625, Universit\'e de Rennes I 
and INRIA Rennes - Bretagne Atlantique, France}
\email{florent.malrieu(at)univ-rennes1.fr}
\urladdr{http://perso.univ-rennes1.fr/florent.malrieu/}

\begin{document}

\begin{abstract}
Motivated by some numerical observations on molecular dynamics
simulations, we analyze metastable trajectories in a very simple
setting, namely paths generated by a one-dimensional overdamped Langevin equation for a double well potential. More precisely, we are interested in so-called reactive paths, namely trajectories which leave definitely one well and reach the other one.
The aim of this paper is to precisely analyze the distribution of the lengths of reactive paths in the limit of small temperature, and to compare the theoretical results to numerical results obtained by a Monte Carlo method, namely the multi-level splitting approach~\cite{CGLP}.
\end{abstract}

\maketitle


\section{Introduction and main results}\label{sec:intro}

\subsection{Motivation and presentation of reactive paths}

A prototypical example of a dynamics which is used to describe the 
evolution of a molecular system is the so-called overdamped Langevin 
dynamics:
\begin{equation}\label{eq:sdee}
dX^{(\varepsilon)}_t=-\nabla V\left(X^{(\varepsilon)}_t\right)\,dt
+\sqrt{2\varepsilon}dB_t,  
\end{equation}
where $X^{(\varepsilon)}_t \in \dR^d$ denotes the position of the particles 
(think of the nuclei of a molecule), $V\,:\, \dR^d\to\dR$ is the given 
potential function modeling the interaction between the particles, 
${(B_t)}_{t\geq 0}$ is a standard Brownian motion on $\dR^d$ 
and $\varepsilon$ is a (small) positive parameter proportional to temperature. 
The potential $V$ is assumed to be smooth
and to grow sufficiently fast to infinity at infinity  so that the stochastic 
differential equation~\eqref{eq:sdee} admits a unique strong solution. 
One common feature of many molecular dynamics simulations is that 
the dynamics~\eqref{eq:sdee} is metastable: the stochastic process 
$\left(X^{(\varepsilon)}_t\right)_{t \ge 0}$ spends a lot of time in some 
region before hopping to another region. These hopping events are exactly 
those of interest, since they are associated to large changes of conformations 
of the molecular system, which can be seen at the macroscopic level.

In the following, we focus on the limit of small temperature (namely $\varepsilon$ goes to zero). In this case, the Freidlin-Wentzell theory \cite{FW} is very useful to understand these hopping events. Specifically, it turns out that the metastable states are neighborhoods of the local minima of the potential $V$, and that the time it takes to leave a metastable state to reach another one is of the order of 
\begin{equation}\label{eq:kramers}
C \exp(\delta V / \epsilon).
\end{equation} Here, $\delta V$ is the height of the barrier to be
overcome (namely the difference in energy between the saddle point and
the initial local minimum), and $C$ is a constant depending on the
eigenvalues of the Hessian of the potential at the minimum and at the
saddle point (see Equation~\eqref{eq:EK} below for a precise formula in the one-dimensional case). This is the so-called Eyring-Kramers (or Arrhenius) law, and we refer for example to~\cite{BEGK,B11,MS} for more precise results.

Actually, the most interesting part of a transition path between two metastable states is the final part, namely the piece of the trajectory which definitely leaves the initial metastable state and then goes to the next metastable region: this is the so-called {\em reactive trajectory} (or {\em reactive path})~\cite{hummer-04,e-vanden-eijnden-04}. In particular, reactive paths give important information on the transition states between the two metastable states. One numerical challenge in molecular dynamics is thus to be able to efficiently sample these reactive paths.  Notice that from the Eyring-Kramers law~\eqref{eq:kramers}, 
a naive Monte Carlo method (generating trajectories according
to~\eqref{eq:sdee} and waiting for a transition event) cannot provide efficiently a large sample of reactive paths, hence the need for dedicated algorithms.

In~\cite{CGLP}, we proposed a numerical method based on an adaptive multilevel splitting algorithm to sample reactive trajectories. One interesting observation we made is that the lengths of these reactive paths seem to behave very differently from~\eqref{eq:kramers}, see Figure~\ref{fi:couleurs} below. It seems that, in the limit of small $\varepsilon$, the distribution of these lengths is a fixed distribution shifted by an additive factor $-\log \epsilon$. The aim of this work is to use analytical tools to precisely analyze this distribution in the asymptotic regime $\varepsilon$ goes to zero, and to give a proof of this numerical observation.

\subsection{The one-dimensional setting and our main results}

In the following, we consider a one-dimensional case ($d=1$), and we
assume (for simplicity)  that the potential $V$ admits exactly two
local minima ($V$ is a double-well potential). More precisely, let us
denote $x^*<y^*$ the two local minima of $V$ and $z^*\in (x^*,y^*)$
the point where $V$ reaches its local maximum in between. As explained above, we are interested in trajectories solution to~\eqref{eq:sdee} from $x^*$ to $y^*$, and more precisely in the end of the path from $x^*$ to $y^*$ (the reactive paths). In order to precisely define these reactive paths, let us introduce the first hitting time of a ball centered at $y^*$ with (small) radius $\delta_y >0$, starting from $x^*$:
\[
T^{x^*}_{y^*}=\inf\BRA{t>0\,:\, |X^{(\varepsilon)}_t - y^*| < \delta_y} 
\quad\text{with }X^{(\varepsilon)}_0=x^*. 
\]
In this setting, formula~\eqref{eq:kramers} writes (notice that $V''(x^*)>0$ and $V''(z^*)<0$): 
\begin{equation}\label{eq:EK}
\dE\PAR{T^{x^*}_{y^*}}\underset{\varepsilon\to 0}{\sim}
\frac{2\pi}{\sqrt{V''(x^*)\ABS{V''(z^*)}}}
\exp\left((V(z^*)-V(x^*))/\varepsilon\right).
\end{equation}
The $d$-dimensional version of this result is established in \cite{BEGK}. Let us also introduce the last exit time from the ball centered at $x^*$ with (small) radius $\delta_x>0$ before the time~$T^{x^*}_{y^*}$ (again starting from $X^{(\varepsilon)}_0=x^*$): 
\[
S^{x^*}_{y^*}=
\sup\BRA{t<T^{x^*}_{y^*}\,:\, |X^{(\varepsilon)}_t- x^*| < \delta_x}.
\]
The question we would like to address is: how long is a reactive path, that is the time $T^{x^*}_{y^*}-S^{x^*}_{y^*}$ as $\varepsilon\to 0$ ?

This question was partially addressed in \cite{FW} where the ball centered around $y^*$ is replaced by the complementary of the domain of attraction of $x^*$ for
the deterministic dynamical system corresponding to~\eqref{eq:sdee} with $\varepsilon=0$. Several papers 
are dedicated to the more subtle situation where points on the boundary of this domain are not attracted to $x^*$. In our simple framework, such a domain is given by $(-\infty,z^*)$ 
(see \cite{MS} for such a study). In \cite{D89,D92,D95}, Day is interested 
in the law of the exit time from a domain containing an unstable equilibrium when 
the diffusion starts on the stable manifold. Thus, even if the laws of the exit times considered in these papers 
are related to the distribution of the lengths of reactive paths we deal
with in the present work, these are different quantities, with different asymptotic behaviors. 

In order to specify our purpose, let us now make our assumptions on the potential~$V$ 
more precise. 
\begin{as}\label{as:V}
The potential $V$ is smooth, has exactly two local minima $x^*<0$ and $y^*>0$ and 
a local maximum $z^*=0$. Moreover, $V'$ is positive on $(x^*,0)$ and 
negative on $(0,y^*)$ and the local maximum at $0$ is assumed to be non-degenerate:
\begin{equation}\label{eq:hypV}
V(0)=0, \quad V'(0)=0,
\quad\text{and}\quad 
V''(0)=-\alpha<0.
\end{equation}
\end{as}

Notice that the potential $V$ is close to $x\mapsto -\alpha x^2/2$ for values of $x$ 
around 0. More precisely, it is easy to show that there exist $K>0$ and $\delta>0$ such that, for all $\ABS{x}<\delta$,
\begin{equation}\label{eq:encadrement}
-\alpha x-K x^2\leq V'(x)\leq -\alpha x+K x^2
\quad\text{and}\quad 
-\frac{\alpha x^2}{2}-\frac{K\ABS{x}^3}{3}
\leq V(x)\leq 
-\frac{\alpha x^2}{2}+\frac{K\ABS{x}^3}{3}.
\end{equation}

\begin{ex}
An example of a potential which satisfies the assumption~\ref{as:V} is
\begin{equation}\label{eq:defV}
V:\ x\mapsto \frac{x^4}{4}-\frac{x^2}{2}.
\end{equation}
In this case, $-1$ and $+1$ are the two (global) minima. This is a 
double well potential with a local maximum at $x=0$ which is non 
degenerate, with $\alpha=1$.
\end{ex}
 
Let us denote $A=x^*+\delta_x \in(x^*,0)$, $B = y^*-\delta_y \in(0,y^*)$ and $x\in (A,0)$. 
We are interested in the behavior of 
\[
T_{x\to B}=\inf\BRA{t>0\,:\,X^{(\varepsilon)}_t=B}
\quad\text{conditionally to the event }
\BRA{X^{(\varepsilon)}_0=x,\ T_B<T_A}
\]
\noindent 
when $\varepsilon$ goes to zero. At the end of the day, the aim is to let $x$ go to $A$. As mentioned above, simulations in \cite{CGLP} suggest that, if the local maximum is non degenerated, then the law of this length looks like a fixed law shifted as $\varepsilon$ goes to $0$. Figure \ref{fi:couleurs} presents the density of the reactive path $T_{x\to B}$ for several values of $\varepsilon$, when $V(x)=x^4/4-x^2/2$, $A=-0.9$, $B=0.9$, and $x=-0.89$. In \cite{CGLP,LIMIT}, it is suggested 
that the asymptotic shape of these laws is an Inverse Gaussian 
distribution. In fact, it is not the case: it turns out to be a Gumbel distribution.
 
\begin{defi}[Standard Gumbel distribution]
 The standard Gumbel distribution is defined by its density function
\[
 f(x)
 =\exp\PAR{-x-e^{-x}}.
\]
 Its Laplace transform is given by 
\[
 \dE\PAR{e^{-s G}}=
\begin{cases}
 \Gamma(1+s)&\text{ if }s> -1,\\
 +\infty &\text{otherwise,}
\end{cases}
\] 
where $\Gamma(z)=\int_0^\infty t^{z-1} e^{-t} \, dt$ is the Euler's Gamma function. 
\end{defi}
The main result of the paper is the following convergence in distribution.  
\begin{thm}\label{th:total}
Under Assumption~\ref{as:V}, for any $A\in (x^*,0)$, $B\in (0,y^*)$, and $x\in(A,0)$
we have, conditionally to the event $\BRA{X^{(\varepsilon)}_0=x,\ T_B<T_A}$,
\[
 T_{x\to B}+\frac{1}{\alpha}\log\varepsilon 
 \xrightarrow[\varepsilon\to 0]{\cL} 
\frac{1}{\alpha}
\PAR{\log(\vert x\vert B)+F(x)+F(B)-\log\alpha+G}
\]
where $G$ is a standard Gumbel random variable and 
\[
F(s)=\int_s^0\!\PAR{\frac{\alpha}{V'(t)}+\frac{1}{t}}\,dt
\]
for any $s\in (x^*,y^*)$.
\end{thm}

Notice that by~\eqref{eq:hypV}, the integral defining the function $F$ is well defined. We slightly abuse notation and denote $T_{A\to B}$ the limit of $T_{x\to B}$ when $x$ goes to $A$. We then have
\[
 T_{A\to B}+\frac{1}{\alpha}\log\varepsilon 
 \xrightarrow[\varepsilon\to 0]{\cL} 
\frac{1}{\alpha}
\PAR{\log(\vert A\vert B)+F(A)+F(B)-\log\alpha+G}.
\]

\begin{ex}
Let us come back to our previous example where the potential $V$ is defined as
\begin{equation}\nonumber
V:\ x\mapsto \frac{x^4}{4}-\frac{x^2}{2}.
\end{equation}
In this case, $\alpha=1$ and if we choose $A=-0.9$, $B=0.9$, and $x=-0.89$, we get  
\[
 T_{-0.89\to 0.9}+\log\varepsilon
 \xrightarrow[\varepsilon\to 0]{\cL} 
\log(0.89\times0.9)-\frac{1}{2}\log(1-0.89^2)-\frac{1}{2}\log(1-0.9^2)+G.
\] 
This is illustrated on the left hand side of Figure \ref{fi:couleurs} and on Figure~\ref{fi:doublepuits} below.
\end{ex}

The paper is organized as follows. Section~\ref{sec:classic} recalls 
classical tools that are used in the proofs. Section~\ref{sec:OU} provides a key 
estimate for the (repulsive) Ornstein-Uhlenbeck process. The proof 
of Theorem~\ref{th:total} is given in Section~\ref{sec:proof}. 
Finally, Section~\ref{sec:degenerate} is devoted to particular 
potentials that are degenerated at the origin ({\it i.e.} $V''(0)=0$) or 
singular ({\it e.g.} $V(x)=-\ABS{x}$).

 
\begin{figure}
\begin{center}
 \includegraphics[scale=.4]{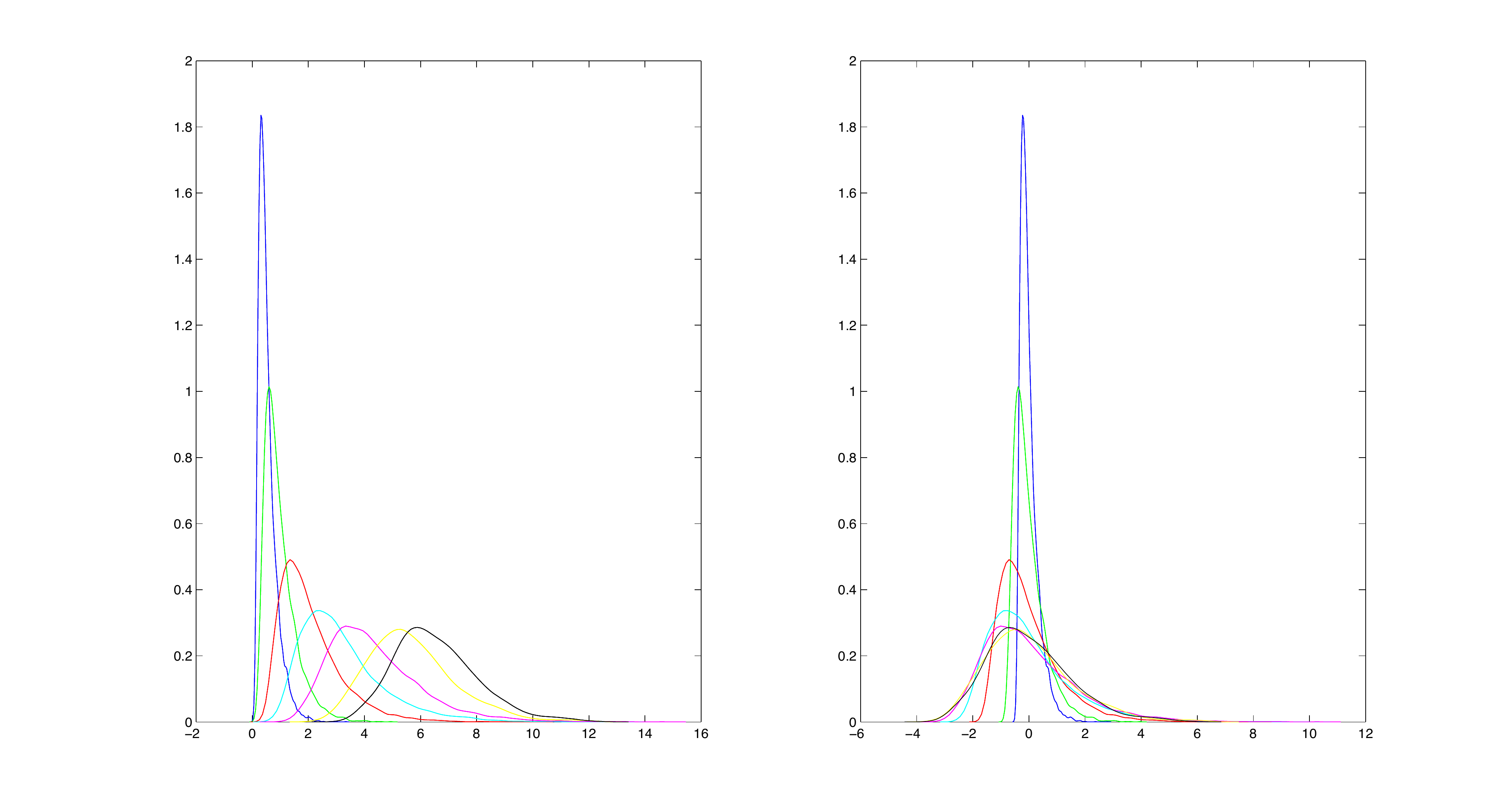}
 \caption{Left : Density of the length $T_{x\to B}$ for different values of 
 $\varepsilon$ (from left to right, $\varepsilon=1,0.5,0.2,0.1,0.05,0.02,0.01$) when $V(x)=x^4/4-x^2/2$, $A=-0.9$, $B=0.9$, and $x=-0.89$. Right : Empirically centered versions of these densities.}
 \label{fi:couleurs}
\end{center}
\end{figure}

\section{Classical tools}\label{sec:classic}

\subsection{Laplace transform of the exit time}\label{sub:Laplace}

Let us first recall how one can link the Laplace transform of the exit time of an 
interval to the infinitesimal generator $A_\varepsilon$ of the diffusion process~\eqref{eq:sdee}
where 
\[
A_\varepsilon f(x)=\varepsilon f''(x)-V'(x)f'(x). 
\]
Fix $a<x<b$ and denote by $H^{(\varepsilon)}_{a,b}$ 
the first exit time from $(a,b)$, starting from $x$:
\[
H^{(\varepsilon)}_{a,b}=\inf\BRA{t>0\,:\, X^{(\varepsilon)}_t\notin(a,b)}
=T^{(\varepsilon)}_a\wedge T^{(\varepsilon)}_b
\quad\text{where}\quad
T^{(\varepsilon)}_c=\inf\BRA{t>0\,:\, X^{(\varepsilon)}_t=c}. 
\]
In the sequel, we may drop the superscript $\varepsilon$ and the indices $a$ and $b$ and simply denote 
 $H$ for $H^{(\varepsilon)}_{a,b}$. 

Notice that 
$\BRA{X^{(\varepsilon)}_H=b}=\BRA{T_b<T_a}$. 
For any $s\in [0,+\infty)$ and $x\in (a,b)$ let us define  
\begin{equation}\label{eq:defF}
F_\varepsilon(s,x):=\dE_x\ \PAR{e^{-s H}\vert X^{(\varepsilon)}_H=b}
\quad\text{and}\quad
F_\varepsilon(s)=
\lim_{x\to a}F_\varepsilon(s,x).
\end{equation}
Let us also introduce the function $u_s$ solution of 
\begin{equation}\label{eq:generalu}
\begin{cases}
 A_\varepsilon u_s(x)=s u_s(x),\quad x\in(a,b),\\
 u_s(a)=0,\quad u_s(b)=1. 
\end{cases}
\end{equation}
Itô's formula ensures that 
${(u_s(X^{(\varepsilon)}_t)e^{-s t})}_{t\geq 0}$
is a martingale and then
\[
u_s(x)=\dE_x\PAR{u_s(X^{(\varepsilon)}_H)e^{-s H}}
=\dE_x\PAR{e^{-s H}\ind_\BRA{X^{(\varepsilon)}_H=b}}. 
\]
Consequently, 
\begin{equation}\label{eq:Feps}
F_\varepsilon(s,x)=\frac{u_s(x)}{u_0(x)}.
\end{equation}
This formula will play a crucial role in the following.

\begin{rem}[The exit distribution]
When $s=0$, Equation \eqref{eq:generalu} is easy to solve: for any 
$x\in (a,b)$, 
\[
u_0(x)=\dP_x(T_b<T_a)=
\frac{\int_a^x\! e^{V(s)/\varepsilon}\,ds}{\int_a^b\! e^{V(s)/\varepsilon}\,ds}.
\] 
\end{rem}

\subsection{The $h$-transform of Doob}
The process ${(X^{(\varepsilon)}_t)}_{t\geq 0}$ solution of the 
stochastic differential equation~\eqref{eq:sdee} conditionally to the event 
$\BRA{T_b<T_a}$ is still a Markov process. Moreover, it can be seen 
as the solution of a modified stochastic differential equation with a drift that 
depends on the exit probabilities for the process. This is the so-called 
\emph{$h$-transform}. 

\begin{prop}\label{prop:h-transform}
Conditionally to the event $\BRA{T_b<T_a}$, the process $X^{(\varepsilon)}$ is a 
diffusion process and it is the solution of
\begin{equation}\label{eq:X-h-transform}
d\bar X^{(\varepsilon)}_t=\sqrt{2\varepsilon}\,dB_t+
\PAR{-V'(\bar X^{(\varepsilon)}_t)+
2\varepsilon \frac{h_\varepsilon'(\bar X^{(\varepsilon)}_t)}
{h_\varepsilon(\bar X^{(\varepsilon)}_t)}
\ind_\BRA{T_b>t}}\,dt
\end{equation}
where, for any $x\in (a,b)$, 
\[
h_\varepsilon(x)=
\frac{\int_a^x\! e^{V(s)/\varepsilon}\,ds}{\int_a^b\! e^{V(s)/\varepsilon}\,ds}.
\]
\end{prop}

See \cite{D92} for the proof of this assertion \emph{via} Girsanov's theorem. 
Similarly, one could write the equation satisfied by a diffusion process
conditioned to reach a given point at a given time (see \cite{Marchand} for 
instance).

\begin{rem}
 The additional drift is singular at point $a$ and is equivalent to $2\varepsilon (x-a)^{-1}$. 
 This ensures that $Y$ cannot hit $a$ as far as $t<T_a$ (see the Feller 
 condition in \cite{RY}).
\end{rem}

Let us associate to a potential $V$ the modified drift induced by 
the $h$-transform on the interval $(a,b)$: 
\begin{equation}\label{eq:new-drift}
 b_V(x)=-V'(x)+2\varepsilon\frac{h'_\varepsilon(x)}{h_\varepsilon(x)}=-V'(x)+2\varepsilon\frac{e^{V(x)/\varepsilon}}{\int_{a}^x\! e^{V(s)/\varepsilon}\,ds}.
\end{equation}
\begin{lem}\label{lem:conv-h-drift}
Let us assume that $x^*<a<0<b<y^*$ and that $V$ satisfies the assumption~\ref{as:V}. Then, for any $x\in (a,b)$, 
\[
  b_V(x)\xrightarrow[\varepsilon\to 0]{} \ABS{V'(x)}. 
\]
\end{lem}

\begin{proof}
 Since $V$ is increasing on $(a,0)$ then, for any $x\in(a,0)$, 
\[
\int_{a}^x e^{V(s)/\varepsilon}\,ds 
\underset{\varepsilon\to 0}{\sim}
\varepsilon\frac{e^{V(x)/\varepsilon}}{V'(x)}
\quad\text{and}\quad
b_V(x)\underset{\varepsilon\to 0}{\sim}V'(x)=\ABS{V'(x)}. 
\]
In other words, the $h$-transform turns the negative drift $-V'(x)$ to 
its opposite. Moreover, it is  obvious that, for any $x>0$, 
$h_\varepsilon(x)$ goes to 1 as $\varepsilon\to 0$ and 
$h_\varepsilon'(x)/h_\varepsilon(x)$ goes to 0 exponentially fast: 
in this case, $b_V(x)\to -V'(x)=\ABS{V'(x)}$. Finally, one can notice that 
\[
b_V(0)= \frac{2\varepsilon}{\int_{a}^0\! e^{V(s)/\varepsilon}\,ds}
\underset{\varepsilon\to 0}{\sim}
\sqrt{\frac{8\ABS{V''(0)}\varepsilon}{\pi}}
\]
since $V(s)\sim V''(0)s^2/2$ when $s$ goes to zero. 
\end{proof}
The $h$-transform and the previous Lemma will be two major ingredients for the arguments below.

In the former proof, and in the following, we constantly use the Laplace's method to get equivalents of integrals in the limit $\varepsilon$ goes to $0$. Let us recall these classical results:
\begin{lem}
Let $[a,b)$ be some interval of $\dR$ (with possibly $b=\infty$), $\psi:[a,b) \to \dR$ a function continuous at point $a$ such that $\psi(a) \neq 0$ and $\varphi:[a,b) \to \dR$ a function of class ${\mathcal C}^2$ such that
$\varphi' < 0 \text{ on }(a,b).$ Let us denote $f(\varepsilon)=\int_a^b \exp(\varphi(x)/\varepsilon)\,  \psi(x) \, dx$. Then, we have:
\begin{itemize}
\item If $\varphi'(a)= 0$ and $\varphi''(a) < 0$, $$f(\varepsilon) \underset{\varepsilon\to 0}{\sim} \sqrt{\frac{\pi \varepsilon}{2|\varphi''(a)|}} \exp(\varphi(a)/\varepsilon) \, \psi(a). $$
\item If $\varphi'(a) < 0$,
$$f(\varepsilon) \underset{\varepsilon\to 0}{\sim} \frac{\varepsilon}{|\varphi'(a)|} \exp(\varphi(a)/\varepsilon) \, \psi(a). $$
\end{itemize}
\end{lem}

\section{Main example: the repulsive Ornstein-Uhlenbeck process}\label{sec:OU}

In this section we deal with the simplest example of a potential that 
is smooth and strictly concave at the origin. We assume here that 
$V(x)=-\alpha x^2/2$ on the set $[-b,b]$ with $b,\alpha>0$ and 
then investigate the behavior of the process:
\begin{equation}\label{eq:OU-gamma}
d Y^{(\varepsilon,\alpha)}_t=
\sqrt{2\varepsilon} dB_t+\alpha Y^{(\varepsilon,\alpha)}_t\,dt.  
\end{equation}
In the sequel we denote 
\[
T^{\varepsilon,\alpha,x}_b=
\inf\BRA{t\geq 0\,:\, Y^{(\varepsilon,\alpha)}_t=b}
\quad\text{with }Y^{(\varepsilon,\alpha)}_0=x\in(-b,b). 
\]
For the sake of simplicity, we first deal with the case $\alpha=1$ and then 
we will get the general result thanks to a straightforward scaling. The 
strategy is to express the Laplace transform of this exit time in terms 
of special functions and then to derive its asymptotic form as 
$\varepsilon$~goes to 0. In the sequel, $T_b$ stands for 
$T^{(\varepsilon,1,x)}_b$.

\begin{prop}\label{prop:laplace}
Let $x \in (-b,b)$. For any $s>-1$, we have 
\begin{equation}\label{eq:prop-Laplace}
\dE_x\PAR{e^{-s T_b}\,\big\vert\, T_b<T_{-b}} 
\underset{\varepsilon\to 0}{\sim}
\begin{cases}
\displaystyle{\Gamma(1+s)e^{-s(-\log \varepsilon+\log b+\log\ABS{x})}} 
&\text{ if }x\in(-b,0), \\
\displaystyle{\frac{2^{s/2}}{\sqrt{\pi}}\Gamma\PAR{\frac{1+s}{2}}
e^{-s(-\log\sqrt\varepsilon+\log b)}} &\text{ if }x=0, \\
\displaystyle{e^{-s (\log b-\log x)}} &\text{ if }x\in(0,b). 
\end{cases}
\end{equation}
\end{prop}

One can also notice that $\lim_{\varepsilon \to 0}\dE_x\PAR{e^{-s T_b}\,\big\vert\, T_b<T_{-b}}= \infty$ if $s \le-1$.

\begin{proof}
The Laplace transform of the exit time is linked by~\eqref{eq:Feps} to the solution $u_s$ of 
\begin{equation}\label{eq:u}
\begin{cases}
 \varepsilon u_s''(x)+xu_s'(x)=s u_s(x),& x\in(-b,b),\\
 u_s(-b)=0,\\
  u_s(b)=1. 
\end{cases}
\end{equation}
Let us define $b_\varepsilon=b/\sqrt\varepsilon$ and the function $v_s$ 
on $(-b_\varepsilon,b_\varepsilon)$ by 
$v_s(y)=u_s(y\sqrt\varepsilon)$. Then $v_s$ is the solution of 
\begin{equation}\label{eq:v}
\begin{cases}
  v_s''(y)+yv_s'(y)=s v_s(y),
  \quad y\in(-b_\varepsilon,b_\varepsilon),\\
 v_s(-b_\varepsilon)=0,\\
  v_s(b_\varepsilon)=1. 
\end{cases}
\end{equation}
As it is recalled in Section~\ref{sub:Laplace} (see~\eqref{eq:Feps}), one has 
\[
\dE_x\PAR{e^{-s T_b}\vert T_b<T_{-b}}
=\frac{u_s(x)}{u_0(x)}
=\frac{v_s(x/\sqrt\varepsilon)}{v_0(x/\sqrt\varepsilon)}.
\]
One can express the function $v_s$ in terms of some 
special functions. Let $\nu>0$ and define the parabolic cylinder function 
$D_{-\nu}$ as  
\[
D_{-\nu}(x)=\frac{1}{\Gamma(\nu)}e^{-x^2/4}
\int_0^\infty\! t^{\nu-1}e^{-t^2/2-xt}\,dt, \quad x\in\dR.
\]
The so-called Whittaker function $D_{-\nu}$ is solution of 
\[
D_{-\nu}''(x)-\PAR{\frac{x^2}{4}+\nu-\frac{1}{2}}D_{-\nu}(x)=0.
\]
See \cite[ch.19]{AS} or \cite[p.639]{BS} for further details. Define the 
function $\varphi_\nu$ by 
\[
\varphi_\nu(x)=e^{-x^2/4}D_{-\nu}(x).
\]
One can check with a straightforward computation that 
\begin{equation}\label{eq:nu}
\varphi_\nu''(x)+x\varphi_\nu'(x)=(\nu-1)\varphi_\nu(x). 
\end{equation}
In the sequel, $s$ and $\nu$ are linked by the relation $$\nu=s+1>0.$$ 
Notice that $\psi_\nu$: $x\mapsto \varphi_\nu(-x)$ is also solution 
of \eqref{eq:nu} (and $\psi_\nu$ and $\varphi_\nu$ are linearly independent). Then,
the solution of \eqref{eq:v} is a linear combination of 
$\varphi_\nu$ and $\psi_\nu$ satisfying the boundary conditions. 
The function $v_s$ is given by 
\begin{equation}\label{eq:v-expression}
v_s(x)=
\frac{\varphi_\nu(-b_\varepsilon)\varphi_\nu(-x)
-\varphi_\nu(b_\varepsilon)\varphi_\nu(x)}
{\varphi_\nu(-b_\varepsilon)^2-\varphi_\nu(b_\varepsilon)^2}.
\end{equation}
Let us study the asymptotic behavior of $\varphi_\nu(b)$ and $\varphi_\nu(-b)$
as $b\to+\infty$. The Laplace's method ensures that 
\[
\int_0^\infty \! t^{\nu-1}e^{-t^2/2} e^{-bt}\,dt 
\underset{b\to+\infty}{\sim}
\frac{\Gamma(\nu)}{b^{\nu}}. 
\]
As a consequence, 
\[
\varphi_\nu(b) \underset{b\to+\infty}{\sim}
\frac{e^{-b^2/2}}{b^{\nu}}.
\] 
Moreover, 
\begin{align*}
 \varphi_\nu(-b)
&=\frac{1}{\Gamma(\nu)}
 \int_{0}^\infty\! t^{\nu-1}e^{-(t-b)^2/2}\,dt\\
&\underset{b\to+\infty}{\sim} \frac{\sqrt{2\pi}}{\Gamma(\nu)} b^{\nu-1}.
\end{align*}
In particular, one obtains that 
\[
\varphi_\nu(-b)^2-\varphi_\nu(b)^2
\underset{b\to+\infty}{\sim}
\varphi_\nu(-b)^2
\underset{b\to+\infty}{\sim}
\frac{2\pi}{\Gamma(\nu)^2} b^{2(\nu-1)}.
\]
Moreover, we get, for any $\gamma\in (0,1)$, that 
\begin{align*}
\varphi_\nu(-b)\varphi_\nu(\gamma b)-\varphi_\nu(b)\varphi_\nu(-\gamma b)
&\underset{b\to+\infty}{\sim}
\frac{\sqrt{2\pi}}{\Gamma(\nu)}b^{\nu-1} 
\frac{e^{-\gamma^2 b^2/2}}{(\gamma b)^{\nu}}
-\frac{\sqrt{2\pi}}{\Gamma(\nu)}(\gamma b)^{\nu-1} 
\frac{e^{-b^2/2}}{b^{\nu}}\\
&\underset{b\to+\infty}{\sim}
\frac{\sqrt{2\pi}}{\Gamma(\nu)} 
\frac{e^{-\gamma^2 b^2/2}}{\gamma ^{\nu}b }.
\end{align*}
As a conclusion
\begin{equation}\label{eq:v-expression-bis}
\frac{\varphi_\nu(-b)\varphi_\nu(\gamma b)-\varphi_\nu(b)\varphi_\nu(-\gamma b)}
{\varphi_\nu(-b)^2-\varphi_\nu(b)^2}
\underset{b\to+\infty}{\sim}
\frac{\Gamma(\nu)}{\sqrt{2\pi}}
\frac{e^{-\gamma^2 b^2/2}}{\gamma ^{\nu}b^{2\nu-1}}=\frac{\Gamma(\nu)}{\sqrt{2\pi}}
\frac{e^{-\gamma^2 b^2/2}}{(\gamma b) ^{\nu}b^{\nu-1}}.
\end{equation}
One can then deduce the asymptotic behavior 
of $v_s$ solution of Equation \eqref{eq:v} at the point 
$x/\sqrt\varepsilon$ (with $x<0$) replacing in Equation \eqref{eq:v-expression-bis} 
$b$ by $b_\varepsilon=b/\sqrt{\varepsilon}$ and $\gamma $ by $-x/b$
with $\gamma \in (0,1)$. Since $\nu=s+1$, this leads to
\[
v_s(x/\sqrt\varepsilon)
\underset{\varepsilon\to 0}{\sim}
\frac{\Gamma(1+s)}{\sqrt{2\pi}} 
\frac{e^{-x^2/(2\varepsilon)}}
{(-x/\sqrt\varepsilon)^{s+1} (b/\sqrt\varepsilon)^s},
\]
and 
\[
\frac{v_s(x/\sqrt\varepsilon)}{v_0(x/\sqrt\varepsilon)}
\underset{\varepsilon\to 0}{\sim}
\frac{\Gamma(1+s)}{{(-x/\sqrt\varepsilon)^{s} (b/\sqrt\varepsilon)^s}}
=\Gamma(1+s)\PAR{\frac{\varepsilon}{\ABS{x}b}}^s.
\]
This is the expression of the Laplace transform in 
Equation \eqref{eq:prop-Laplace} when $x\in (-b,0)$. 
The two other cases are easier to deal with. If $x=0$, one has (since $v_0(0)=1/2$)
\[
\frac{v_s(0)}{v_0(0)}
\underset{\varepsilon \to 0}{\sim}
\frac{2\varphi_\nu(0)}{\varphi_\nu(-b_\varepsilon)}
\underset{\varepsilon\to 0}{\sim}
\sqrt{\frac{2}{\pi}}
\frac{1}{b_\varepsilon^s}\int_0^{+\infty}\! t^se^{-t^2/2}\,dt
=\frac{2^{s/2}}{\sqrt{\pi}}\Gamma\PAR{\frac{1+s}{2}}
\frac{1}{b_\varepsilon^s}.
\]
At last, if $x=\gamma b$ with $\gamma\in(0,1)$ then 
\[
\frac{v_s(\gamma b_\varepsilon)}{v_0(\gamma b_\varepsilon)}
\underset{\varepsilon\to 0}{\sim}
\frac{\varphi_\nu(-\gamma b_\varepsilon)}{\varphi_\nu(-b_\varepsilon)}
=\gamma^s
=\PAR{\frac{x}{b}}^s.
\]
\end{proof}

\begin{rem}
The parabolic cylinder functions $D_{-\nu}$ also appear in \cite{MR0212865}, Section 2, where the author studies the first exit time from a square root boundary for the Brownian motion.
\end{rem}

Proposition \ref{prop:laplace} yields the following convergence in 
distribution. 
\begin{thm}\label{th:Gumbel}
Let $\alpha>0$ and $x\in(-b,b)$. Conditionally to the event
$\BRA{T^{(\varepsilon,\alpha,x)}_b<T^{(\varepsilon,\alpha,x)}_{-b}}$, 
we have
\begin{equation}\label{eq:decomp-Hx}
T^{(\varepsilon,\alpha,x)}_b
\underset{\varepsilon\to 0}{\overset{\cL}{\sim}}
\frac{1}{\alpha}
\begin{cases}
-\log \varepsilon + \log(\ABS{x} b)
+G -\log \alpha &\text{if }x\in(-b,0),\\
-\log\sqrt\varepsilon+\log b+\tilde G-\log\sqrt\alpha  & \text{if }x=0,\\
\log b-\log x &\text{if }x\in(0,b),
\end{cases}
\end{equation}
where the law of $G$ is the standard Gumbel distribution and $\tilde G$ is a random 
variable with Laplace transform given by 
\[
\dE\PAR{e^{-s \tilde G}}=
\begin{cases}
\displaystyle{\frac{2^{s/2}}{\sqrt{\pi}}\Gamma\PAR{\frac{1+s}{2}}}
&\text{ if }s>-1,\\
+\infty &\text{otherwise}. 
\end{cases}
\]
\end{thm}

\begin{proof}
The case $\alpha=1$ is a straightforward consequence of 
Proposition~\ref{prop:laplace}. Moreover, for any positive 
constants $\tau$ and $\sigma$ one has, for any $t\geq 0$ 
\begin{align*}
\sigma Y^{(\varepsilon,\alpha)}_{\tau t}
&=\sigma Y^{(\varepsilon,\alpha)}_0
+\sigma \sqrt{2\varepsilon}B_{\tau t}
+\sigma\alpha \int_0^{\tau t}\! Y^{(\varepsilon,\alpha)}_s\,ds\\
&\overset{\cL}{=}\sigma Y^{(\varepsilon,\alpha)}_0+
\sqrt{\tau\sigma^2}\sqrt{2\varepsilon}B_t+
\alpha \tau\int_0^t\! \sigma Y^{(\varepsilon,\alpha)}_{\tau u}\,du.
\end{align*}
This ensures that if $\sigma=\sqrt{\alpha}$ and $\tau=1/\alpha$, then 
the process ${(\sigma Y^{(\varepsilon,\alpha)}_{\tau t})}_{t\geq 0}$ 
is solution of Equation~\eqref{eq:OU-gamma} with $\alpha=1$ and 
the initial condition $\sigma Y^{(\varepsilon,\alpha)}_0$. In particular, 
\[
\cL\PAR{T^{(\varepsilon,\alpha,x)}_b\,\vert\, 
T^{(\varepsilon,\alpha,x)}_b<T^{(\varepsilon,\alpha,x)}_{-b}}
=\cL\PAR{\alpha^{-1} T^{(\varepsilon,1,x/\sqrt\alpha)}_{b/\sqrt\alpha}
\,\vert\, T^{(\varepsilon,1,x/\sqrt\alpha )}_{b/\sqrt\alpha}
<T^{(\varepsilon,1,x/\sqrt\alpha)}_{-b/\sqrt\alpha}}.
\]
The result for $\alpha \neq 1$ is then a straightforward consequence of the result for $\alpha=1$. 
\end{proof}

\begin{figure}
\begin{center}
 \includegraphics[scale=0.3]{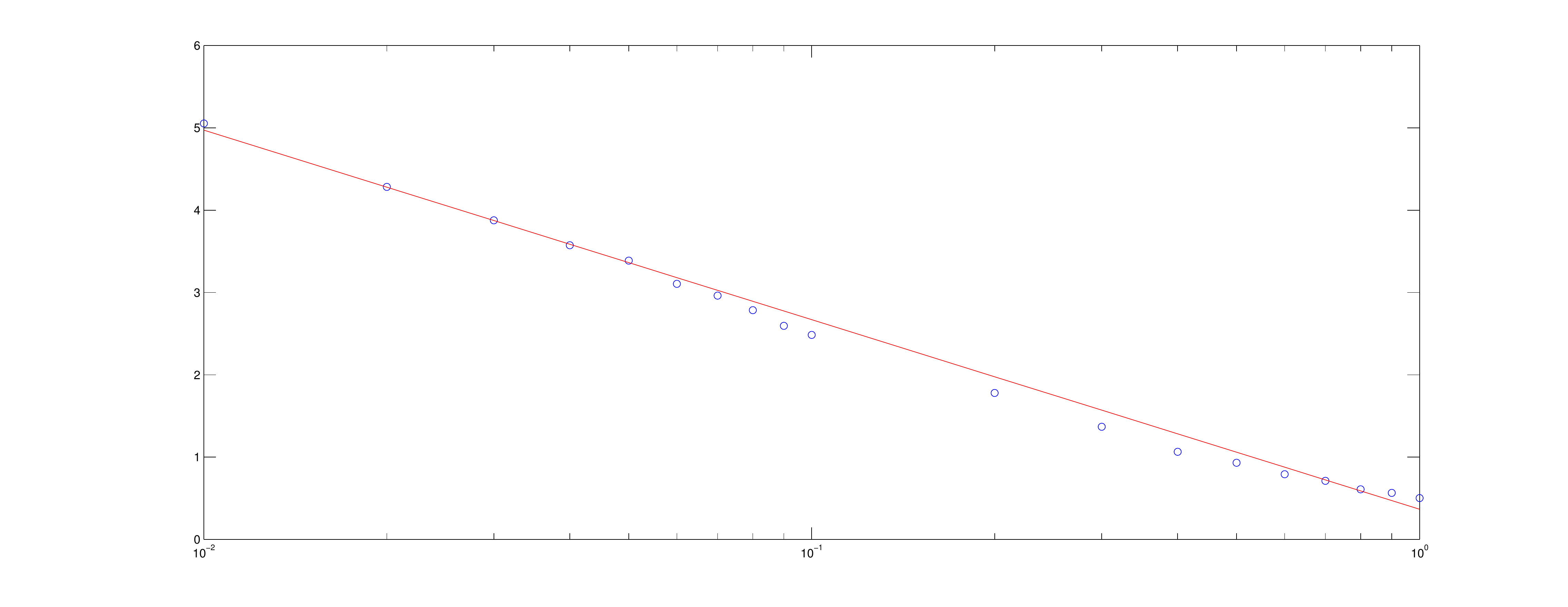}
 \caption{Mean length of the reactive path for the repulsive 
 Ornstein-Uhlenbeck process $d Y^{(\varepsilon)}_t=
\sqrt{2\varepsilon} dB_t+Y^{(\varepsilon)}_t\,dt$, with $Y^{(\varepsilon)}_0=-0.89$, on the set $[-0.9,0.9]$ as a function of $\log\varepsilon$ 
 (see Equation \eqref{eq:decomp-Hx}). The 95\% confidence intervals
 are of the size of the points. The function
 $\log\varepsilon\mapsto-\log\varepsilon+\log(|-0.89|\times
 0.9)+\gamma$ is drawn in dotted line. These results have been obtained with the algorithm described in \cite{CGLP}.}
 \label{fi:OU}
\end{center}
\end{figure}
Notice that the formulas~\eqref{eq:decomp-Hx} admit a limit when
$x$ goes to $-b$. Before coming back to the general case, let us conclude this section with a few remarks about the case of the Ornstein-Uhlenbeck process. 
\begin{rem}
Let us discuss the asymptotic behavior~\eqref{eq:decomp-Hx} of the 
length of the reactive path when $x\in(-b,0)$ and $\varepsilon$ goes to 0, taking for simplicity $\alpha=1$.
The time $\log(b/\sqrt\varepsilon)$ is the time needed by the deterministic 
process $Y^{(0,1)}$ to go from $\sqrt\varepsilon$ to $b$ since 
$Y^{(0,1)}_t=e^t\sqrt\varepsilon$. The Freidlin-Wentzell 
theory tells us that the first part of the reactive path (from $x$ to 
$-\sqrt\varepsilon$) has a similar length $\log(|x|/\sqrt\varepsilon)$. Finally, the Gumbel variable $G$ accounts
for the (asymptotic) random time needed by $Y^{(\varepsilon,1)}$ to go from 
$-\sqrt\varepsilon$ to $\sqrt\varepsilon$.  
\end{rem}
\begin{rem}\label{rem:laplace}
It is easy to check from the proof that the results of 
Proposition~\ref{prop:laplace} are still valid if $b=b_\varepsilon$ and $x=x_\varepsilon$ 
depend on $\varepsilon$ as long as $b_\varepsilon/\sqrt{\varepsilon}$ and $x_\varepsilon/\sqrt{\varepsilon}$
go to infinity when $\varepsilon$ goes to zero. For example, if $b_\varepsilon >0$ is such that $\lim_{\varepsilon \to 0} b_\varepsilon/\sqrt{\varepsilon} = \infty$ and $x_\varepsilon \in (-b_\varepsilon,0)$ is such that $\lim_{\varepsilon \to 0} x_\varepsilon/\sqrt{\varepsilon} = -\infty$, then $T^{(\varepsilon,\alpha,x_{\varepsilon})}_{b_{\varepsilon}}
\underset{\varepsilon\to 0}{\overset{\cL}{\sim}}
\frac{1}{\alpha}
\left(-\log \varepsilon + \log(\ABS{x_{\varepsilon}} b_{\varepsilon})
+G -\log \alpha \right)$.
This remark will 
be useful in Section \ref{sec:comp}.
\end{rem}
\begin{rem}
Figure \ref{fi:OU} illustrates Theorem \ref{th:Gumbel} for the repulsive 
 Ornstein-Uhlenbeck process $d Y^{(\varepsilon)}_t=
\sqrt{2\varepsilon} dB_t+Y^{(\varepsilon)}_t\,dt$, with $Y^{(\varepsilon)}_0=-0.89$, on the set $[-0.9,0.9]$. Denoting $T_{-0.89\to 0.9}$ the length of the reactive path from $-0.89$ to $0.9$, then Equation (\ref{eq:decomp-Hx}) ensures that 
$\dE[T_{-0.89\to 0.9}]$ is equivalent to $-\log\varepsilon+\log(|-0.89|\times 0.9)+\gamma$, when $\varepsilon$ goes to zero ($\gamma$ stands here for the Euler's constant). Figure \ref{fi:OU} compares this theoretical result with the empirical means obtained thanks to the algorithm described in \cite{CGLP} for $\varepsilon$ ranging from $0.01$ to $1$.
\end{rem}

\section{The general (strictly convex) case}\label{sec:proof}

Let us now come back to the general strictly convex case described in Section \ref{sec:intro}. We recall the notation. The potential $V$ has exactly two local minima $x^*<0$ and $y^*>0$ and 
a local maximum $z^*=0$. Moreover, $V'$ is positive on $(x^*,0)$ and 
negative on $(0,y^*)$ and 
$$V(0)=0, \quad V'(0)=0,
\quad\text{and}\quad 
V''(0)=-\alpha<0.$$
Let us consider $A\in(x^*,0)$, $B\in(0,y^*)$ and $x\in (A,0)$. 
We are interested in the behavior of 
\[
T_{x\to B}=\inf\BRA{t>0\,:\,X^{(\varepsilon)}_t=B}
\quad\text{conditionally to the event }
\BRA{X^{(\varepsilon)}_0=x,\ T_B<T_A}
\]
\noindent 
when $\varepsilon$ goes to zero.

According to the Markov property, and considering the initial point $x\in(A,0)$, the strategy is to decompose the reactive path from $x$ to $B$ into three independent pieces: 
\begin{equation}\label{eq:decomp}
 H=T_{x\to -c_\varepsilon}+T_{-c_\varepsilon\to b_\varepsilon}
+T_{ b_\varepsilon \to B}
 \end{equation}
on the event $\BRA{T_B<T_A}$ where 
$0<c_\varepsilon<b_\varepsilon<\vert x\vert\wedge B$ will be chosen 
in the sequel. More precisely, we will choose 
\[
b_\varepsilon=\varepsilon^\beta
\quad\text{and}\quad
c_\varepsilon=\varepsilon^\gamma
\quad\text{with}\quad
\frac{2}{5}<\beta<\gamma<\frac{1}{2}.
\]
The first and third times in \eqref{eq:decomp} are essentially deterministic, as specified by the following result.

\begin{prop}\label{prop:up-down}
If $0<\beta,\gamma<1/2$, then, conditionally to the event $\BRA{T_B<T_A}$, 
\[
T_{b_\varepsilon \to B}-t_{ b_\varepsilon \to B}\xrightarrow[\varepsilon\to 0]{\dP} 0
\quad \text{and}\quad 
T_{ x \to -c_\varepsilon}- t_{-c_\varepsilon\to x}\xrightarrow[\varepsilon\to 0]{\dP} 0,
\]
where $t_{b_\varepsilon \to B}$ is the time for the unnoised process to 
reach $B$ from $b_\varepsilon\in(0,B)$: 
\[
t_{ b_\varepsilon \to B}=-\int_{b_\varepsilon}^B\!\frac{1}{V'(s)}\,ds,
\]
and $t_{-c_\varepsilon\to x}$ is the time for the unnoised process to 
reach $x$ from $-c_\varepsilon\in(x,0)$: 
\[
t_{-c_\varepsilon\to x}=-\int_{-c_\varepsilon}^x\!\frac{1}{V'(s)}\,ds.
\]
\end{prop}

This is proved in Sections \ref{sec:down} and \ref{sec:up}. In Section \ref{sec:comp} we 
compare the second time in~\eqref{eq:decomp} to the reactive time of an Ornstein-Uhlenbeck 
process. 

\subsection{Going down is easy}\label{sec:down}

The easiest part is to study the third time interval $T_{ b_\varepsilon \to B}$. Our goal here is to prove that, starting at $b_\varepsilon$, the process $X^{(\varepsilon)}$ is close to the deterministic path ${(x_t)}_{t\geq 0}$ solution of the ordinary differential equation 
\begin{equation}\nonumber
\begin{cases}
\dot{x}_t=-V'(x_t) & t\geq 0, \\
x_0=b_\varepsilon.
\end{cases}
\end{equation}
 In this aim, we need to state a few intermediate results. First, it is readily seen that, starting at $b_\varepsilon$, the probability for the process $(X^{(\varepsilon)}_t)_{t\geq 0}$ to hit 0 before $B$ goes to 0 at an exponential rate when $\varepsilon$ goes to 0. Indeed, since $b_\varepsilon=\varepsilon^\beta$ with $\beta<1/2$, we have
$$\dP_{b_\varepsilon}(T_0<T_B)=\frac{\int_{b_\varepsilon}^Be^{V(s)/\varepsilon}ds}{\int_{0}^Be^{V(s)/\varepsilon}ds}\underset{\varepsilon\to 0}{\sim}\sqrt{\frac{2 \alpha\varepsilon}{\pi |V'(b_\epsilon)|}}\ e^{V(b_\varepsilon)/\varepsilon}.$$
In the following, we will denote $\Omega_\varepsilon$ the event on which this does not occur, so that $\dP(\Omega_\varepsilon)$ goes to 1 when $\varepsilon$ goes to 0. 

Of course, this will also be true for the event $\Omega_x$ which is defined as: the process starts at a fixed point $x\in(0,B)$ (independent of $\varepsilon$) and does not hit 0 before $B$. Again, $\dP(\Omega_x)$ goes to 1 when $\varepsilon$ goes to 0.  Then, starting at $x\in(0,B)$, our aim is to compare the deterministic path ${(x_t)}_{t\geq 0}$ solution of the ordinary differential equation 
\begin{equation}\label{eq:xd}
\begin{cases}
\dot{x}_t=-V'(x_t)  &t\geq 0, \\
x_0=x,
\end{cases}
\end{equation} 
and the random process 
\begin{equation}\nonumber
\begin{cases}
dX^{(\varepsilon)}_t=-V'(X^{(\varepsilon)}_t)\,dt+\sqrt{2\varepsilon}dB_t &\ t\geq 0, \\
X^{(\varepsilon)}_0=x.
\end{cases}
\end{equation}
For this, let us introduce $c\in(B,y^*)$ such that $c-B<B-x$, the deterministic time $t_c=t_{x\to c}=\inf\BRA{t>0\ :\ x_t=c}$ and the stochastic time $T_c=T_{x\to c}=\inf\BRA{t>0\ :\ X^{(\varepsilon)}_t=c}$. 
\begin{lem}\label{lem:lim-easy}
Define $K:=\sup_{s\in[0,c]}\ABS{V''(s)}$, then 
\[
\dP\PAR{ \Omega_x \cap \left\{
\sup_{0\leq s\leq t_c\wedge T_c}
\vert X^{(\varepsilon)}_s-x_s\vert \geq \eta \right\}}
\leq 2\exp\PAR{-\frac{\eta^2 e^{-2K t_c}}{4\varepsilon t_c}}.
\]
\end{lem}
\begin{proof}
Let us assume that we work on the event $\Omega_x$. For any $t\leq t_c\wedge T_c$,  
\[
X^{(\varepsilon)}_t-x_t=
-\int_0^t\! (V'(X^{(\varepsilon)}_s)-V'(x_s))\,ds
+\sqrt{2\varepsilon}B_t. 
\]
The Gronwall Lemma ensures that 
\[
\sup_{0\leq s\leq t_c\wedge T_c}\vert X^{(\varepsilon)}_s-x_s\vert 
\leq \sqrt{2\varepsilon}e^{K (t_c\wedge T_c)}
\sup_{0\leq s\leq t_c\wedge T_c}\vert B_s\vert\leq \sqrt{2\varepsilon}e^{K t_c}
\sup_{0\leq s\leq t_c}\vert B_s\vert. 
\]
Finally, the reflection principle for the Brownian motion ensures that 
$\sup_{0\leq s\leq t}B_s$ has the law of $\ABS{B_t}$. As a consequence, 
for any $r,t\geq 0$, 
\[
\dP \left(\sup_{0\leq s\leq t}\ABS{B_s}\geq r \right)
\leq 2 \dP \left(\sup_{0\leq s\leq t}B_s\geq r\right)=2 \dP(\ABS{B_t}\geq r)= 4\dP(B_t\geq r)\leq 2 e^{-r^2/(2t)}.
\]
This concludes the proof.
\end{proof}

The first consequence of this result is that the stochastic time $T_{x\to B}$ 
required by the random process to go from $x\in(0,B)$ to $B$ converges to the deterministic time $t_{x\to B}$ as $\varepsilon\to 0$. 
\begin{cor}\label{cor:lim-easy}
 Let $0<x<B$, then 
 \[
T_{x\to B}\xrightarrow[\varepsilon \to 0]{\dP}t_{x\to B}.
\]
\end{cor}

\begin{proof}
Since we will apply the result of the previous lemma, we still work on the event $\Omega_x$. Let us denote $\eta$ a real number such that $0<\eta<c-B$. Then, on the event
$\Omega_x \cap\BRA{\sup_{0\leq s\leq t_c\wedge T_c}
\vert X^{(\varepsilon)}_s-x_s\vert \leq \eta}$, 
the random time $T_{x\to B}$ belongs to the deterministic interval $[t_{x\to B-\eta}, t_{x\to B+\eta}]$. In other words, 
\[
\int_{B-\eta}^{B}\!\frac{ds}{V'(s)}=-t_{B-\eta\to B}
\leq T_{x\to B}-t_{x\to B} \leq 
t_{B\to B+\eta}=-\int_{B}^{B+\eta}\!\frac{ds}{V'(s)}.
\]
As a consequence, for any $\eta\in(0,c-B)$, 
\[
\ABS{T_{x\to B}-t_{x\to B}} \leq \eta\times\sup_{s\in[B-\eta,B+\eta]}\frac{1}{\ABS{V'(s)}}.
\]
Finally, for any $\eta\in(0,c-B)$, 
\[
\dP\PAR{\Omega_x \cap \BRA{\ABS{T_{x\to B}-t_{x\to B}}\geq \eta\times\sup_{s\in[B-\eta,B+\eta]}\frac{1}{\ABS{V'(s)}}}}
\leq 2\exp\PAR{-\frac{\eta^2 e^{-2K t_c}}{4\varepsilon t_c}},
\]
where $t_c=t_{x\to c}$. This concludes the proof of the corollary.
\end{proof}

Our next goal is to prove that this result still holds if the starting point, namely $b_\varepsilon=\varepsilon^\beta$, goes to 0 sufficiently slowly as $\varepsilon \to 0$, that means if $\beta<1/2$. Let us fix $D\in(b_\varepsilon,B)$ (for sufficiently small $\varepsilon$) such that  
$$\sup_{s\in[0,D]}\ABS{V''(s)}<\frac{\alpha}{2\beta}.$$
This is always possible since $\beta<1/2$, $V''(0)=\alpha$, and $V$ is assumed to be smooth. Then, as previously, we fix $c\in(D,B)$ such that $c-D<D-b_\varepsilon$, and  
$$K:=\sup_{s\in[0,c]}\ABS{V''(s)}<\frac{\alpha}{2\beta}.$$

\begin{cor}\label{cor:lim-easy-bis}
 If $0<b_\varepsilon=\varepsilon^\beta$ with $\beta<1/2$, then 
\[
\ABS{T_{b_\varepsilon\to D}-t_{b_\varepsilon\to D}}
\xrightarrow[\varepsilon \to 0]{\dP}0.
\]
\end{cor}

\begin{proof} 
Here, we work on the event $\Omega_\varepsilon$, which is not a problem since, as mentioned above, $\dP(\Omega_\varepsilon)$ goes to 1 as $\varepsilon$ goes to zero. The first part of the proof is similar to the ones of Lemma~\ref{lem:lim-easy} and Corollary~\ref{cor:lim-easy}. For any $\eta\in(0,c-D)$, 
\[
\dP\PAR{\Omega_\varepsilon \cap \BRA{\ABS{T_{b_\varepsilon\to D}-t_{b_\varepsilon\to D}}\geq \eta\times\sup_{s\in[D-\eta,D+\eta]}\frac{1}{\ABS{V'(s)}}}}
\leq 2\exp\PAR{-\frac{\eta^2 e^{-2K t_c}}{4\varepsilon t_c}},
\]
where $t_c=t_{b_\varepsilon\to c}$. Moreover, since $V'(s)\sim_{s\to 0}-\alpha s$, we have
\[
t_{b_\varepsilon\to c}=-\int_{b_\varepsilon}^c\!\frac{ds}{V'(s)}
=-\frac{1}{\alpha}\log b_\varepsilon+O_\varepsilon(1).
\]
As a consequence, 
\[
\frac{e^{-2K t_c}}{4\varepsilon t_c}
\underset{\varepsilon\to 0}{\sim} 
\frac{\alpha \varepsilon^{\frac{2K \beta}{\alpha}-1}}{-4\beta \log \varepsilon}\xrightarrow[\varepsilon \to 0]{}+\infty. 
\]
This proves the convergence in probability of $T_{b_\varepsilon\to D}$ as $\varepsilon$ goes to 0. 
\end{proof}

Finally, according to the Markov property, we can summary the previous results by decomposing the path from $b_\varepsilon$ to $B$ into two independent pieces: 
$$T_{ b_\varepsilon \to B}=T_{ b_\varepsilon \to D}+T_{ D \to B}.$$
 Using Corollaries~\ref{cor:lim-easy} and~\ref{cor:lim-easy-bis}, we immediately get the following proposition.
\begin{prop}
 If $0<b_\varepsilon=\varepsilon^\beta$ with $\beta<1/2$, then 
\[
\ABS{T_{b_\varepsilon\to B}-t_{b_\varepsilon\to B}}
\xrightarrow[\varepsilon \to 0]{\dP}0
\quad\text{where}\quad 
t_{b_\varepsilon\to B}=-\int_{b_\varepsilon}^B\!\frac{ds}{V'(s)}
\]
\end{prop}

\begin{rem}
If $V$ is given by \eqref{eq:defV}, one can compute the expression of the 
solution ${(x_t)}_{t\geq 0}$ of \eqref{eq:xd}. Let us define the function 
$\Psi$ on $(0,1)$ by 
\[
\Psi(x)=\log\PAR{\frac{x}{\sqrt{1-x^2}}}.
\]
Notice that 
\[
\Psi'(x)=\frac{1}{x}-\frac{1/2}{x-1}-\frac{1/2}{x+1}=-\frac{1}{V'(x)}.
\]
As a consequence, the derivative of $t\mapsto \Psi(x_t)$ is equal to 1 and 
\[
x_t=\Psi^{-1}(\Psi(x)+t). 
\]
Moreover the elapsed time from $x\in(0,B)$ to $B\in(0,1)$ is given by 
\[
t_{x\to B}=\Psi(B)-\Psi(x)
=-\log(x)+\Psi(B)+\frac{1}{2}\log(1-x^2). 
\]
As was just proved, this result still holds when $x=b_\varepsilon$ as long as $0<\beta<1/2$.
\end{rem}

\subsection{The climbing period}\label{sec:up}

\begin{prop}
  If $c_\varepsilon=\varepsilon^\gamma$ with $\gamma<1/2$, then conditionally to the event $\BRA{T_{-c_\varepsilon}<T_A}$, and for $x \in (A,0)$,
\[
\ABS{T_{x\to -c_\varepsilon}- t_{ -c_\varepsilon\to x}}
\xrightarrow[\varepsilon \to 0]{\dP}0
\quad\text{where}\quad 
 t_{-c_\varepsilon\to x}=-\int_{-c_\varepsilon}^x\!\frac{ds}{V'(s)}.
\]
\end{prop}

\begin{proof}
One has to consider the $h$-transformed process and use the fact 
that the new drift converges to $V'(s)$ uniformly on $[A+\delta,-c_\varepsilon]$
with small $\delta$ as $\varepsilon$ goes to $0$, see Lemma~\ref{lem:conv-h-drift} above. 
\end{proof}


\subsection{Central behavior}\label{sec:comp}
Let us finally study the behavior of $T_{-c_\varepsilon\to b_\varepsilon}$
conditionally to the event $\BRA{T_{b_\varepsilon}<T_A}$. 

The sketch of proof is as follows:
\begin{enumerate}
\item Prove that one may assume that the process does not go below $-b_\varepsilon$;
\item Rescale space to change $(-b_\varepsilon,b_\varepsilon)$ to $(-1,1)$;
\item Consider the $h$-transform process to get the evolution of the process 
conditioned on $\BRA{T_1<T_{-1}}$;
\item Introduce the $h$-transformed repulsive Ornstein-Uhlenbeck process; 
\item Compare the drifts;
\item Use Theorem 4.3 in \cite{D92}; 
\item Conclude.\\
\end{enumerate}

\noindent
\textbf{Step 1.}
The first step is to notice that it is equivalent to look 
at $T_{-c_\varepsilon\to b_\varepsilon}$ conditionally to
$\BRA{T_{b_\varepsilon}<T_{-b_\varepsilon}}$ or conditionally to $\BRA{T_{b_\varepsilon}<T_A}$.

\begin{lem}
If $0<\beta<\gamma<1/2$, there exists a constant $C>0$ such that, for any $s>0$
\[
1-C\varepsilon^{\gamma-\beta}\leq 
\frac{\dE_{-c_\varepsilon}
\PAR{e^{-s H_{A,b_\varepsilon}}\vert T_{b_\varepsilon}<T_A}}
{\dE_{-c_\varepsilon}
\PAR{e^{-s H_{-b_\varepsilon,b_\varepsilon}}
\vert T_{b_\varepsilon}<T_{-b_\varepsilon}}}
\leq 1+C\varepsilon^{\gamma-\beta}.
\]
\end{lem}

\begin{proof}
By continuity, 
\[
\BRA{T_{b_\varepsilon}<T_A}=
\BRA{T_{b_\varepsilon}<T_{-b_\varepsilon}}
\cup 
\BRA{T_{-b_\varepsilon}<T_{b_\varepsilon}<T_A}
\]
where the two sets on the right hand side are disjoints. Moreover, 
the strong Markov property ensures that, for any $s\geq 0$, 
\begin{align*}
0&\leq 
\dE_{-c_\varepsilon}
\PAR{e^{-s H_{A,b_\varepsilon}}
\ind_\BRA{T_{-b_\varepsilon}<T_{b_\varepsilon}<T_A}}\\
&\leq 
\dE_{-b_\varepsilon}
\PAR{e^{-s H_{A,b_\varepsilon}}
\ind_\BRA{T_{b_\varepsilon}<T_A}}\\
&\leq 
\dP_{-b_\varepsilon}(T_{-c_\varepsilon}<T_A)
\dE_{-c_\varepsilon}
\PAR{e^{-s H_{A,b_\varepsilon}}
\ind_\BRA{T_{b_\varepsilon}<T_A}}.
\end{align*}
As a consequence, for any $s\geq 0$,
\[
1\leq 
\frac{\dE_{-c_\varepsilon}
\PAR{e^{-s H_{A,b_\varepsilon}}\ind_\BRA{T_{b_\varepsilon<T_A}}}}
{\dE_{-c_\varepsilon}
\PAR{e^{-s H_{-b_\varepsilon,b_\varepsilon}}
\ind_\BRA{T_{b_\varepsilon<T_{-b_\varepsilon}}}}
}\leq 1+\dP_{-b_\varepsilon}(T_{-c_\varepsilon}<T_A) \le
\frac{1}{1-\dP_{-b_\varepsilon}(T_{-c_\varepsilon}<T_A)}.
\]
Making $s=0$ in this equation leads to
\[
1-\dP_{-b_\varepsilon}(T_{-c_\varepsilon}<T_A)\leq \frac{\dP_{-c_\varepsilon}(T_{b_\varepsilon}<T_{-b_\varepsilon})}{\dP_{-c_\varepsilon}(T_{b_\varepsilon}<T_{A})}\leq 1.
\]
Consequently
\[
1-\dP_{-b_\varepsilon}(T_{-c_\varepsilon}<T_A)\leq 
\frac{\dE_{-c_\varepsilon}
\PAR{e^{-s H_{A,b_\varepsilon}}\vert T_{b_\varepsilon}<T_A}}
{\dE_{-c_\varepsilon}
\PAR{e^{-s H_{-b_\varepsilon,b_\varepsilon}}
\vert T_{b_\varepsilon}<T_{-b_\varepsilon}}}
\leq 1+\dP_{-b_\varepsilon}(T_{-c_\varepsilon}<T_A).
\]
To conclude, one just has to remark that, since $V(-b_\varepsilon)\leq V(-c_\varepsilon)$,
\[
\dP_{-b_\varepsilon}(T_{-c_\varepsilon}<T_A)
=\frac{\int_A^{-b_\varepsilon}\! e^{V(s)/\varepsilon}\,ds}
{\int_A^{-c_\varepsilon}\! e^{V(s)/\varepsilon}\,ds}
\underset{\varepsilon\to 0}{\sim}\varepsilon^{\gamma-\beta}e^{(V(-b_\varepsilon)-V(-c_\varepsilon))/\varepsilon}\leq\varepsilon^{\gamma-\beta}.
\]

\end{proof}

\begin{cor}
 If $0<\beta<\gamma<1/2$, then, starting from $-c_\varepsilon$,  
\[
\cL\PAR{H_{A,b_\varepsilon}\vert T_{b_\varepsilon}<T_{A}}
\underset{\varepsilon\to 0}{\sim}
\cL\PAR{H_{A,b_\varepsilon}\vert T_{b_\varepsilon}<T_{-b_\varepsilon}}.
\]
\end{cor}

\noindent
\textbf{Step 2.} Let us define $\eta_\varepsilon=\varepsilon/b_\varepsilon^2$ and the process $Y$ by $Y_t=X^{(\varepsilon)}_t/b_\varepsilon$ (dropping for simplicity the explicit dependence on $\varepsilon$ in the notation for $Y$). Obviously, if 
$X^{(\varepsilon)}_0$ is equal to $-c_\varepsilon$, then $Y$ is solution of 
\[
\begin{cases}
\displaystyle{
dY_t=\sqrt{2\eta_\varepsilon}dB_t-
\frac{V'(b_\varepsilon Y_t)}{b_\varepsilon}\,dt},\\
Y_0=-c_\varepsilon/b_\varepsilon.
\end{cases}
\]
In terms of $Y$, we are interested in the hitting time of $1$ conditionally to the event $\BRA{T_1<T_{-1}}$.\\

\noindent
\textbf{Step 3.}
Thanks to the $h$-transform of Doob, one can see $Y$, conditionally to the event $\BRA{T_1<T_{-1}}$, as a diffusion process. Define, for any $y\in(-1,1)$,  
\[
h_\varepsilon(y)=
\frac{\int_{-1}^y\! e^{V(b_\varepsilon s)/\varepsilon}\,ds}
{\int_{-1}^1\! e^{V(b_\varepsilon s)/\varepsilon}\,ds}
\quad\text{and}\quad 
\frac{h_\varepsilon'(y)}{h_\varepsilon(y)}
=\frac{e^{V(b_\varepsilon y)/\varepsilon}}
{\int_{-1}^y\! e^{V(b_\varepsilon s)/\varepsilon}\,ds}.
\]
Conditionally to $\BRA{T_1<T_{-1}}$, the process $Y$ is solution of 
\[
\begin{cases}
\displaystyle{
dY_t=\sqrt{2\eta_\varepsilon}dB_t+\PAR{
-\frac{V'(b_\varepsilon Y_t)}{b_\varepsilon}+2\eta_\varepsilon 
\frac{h_\varepsilon'(Y_t)}{h_\varepsilon(Y_t)}\ind_\BRA{t\leq T_1}}\,dt
},\\
Y_0=-c_\varepsilon/b_\varepsilon.
\end{cases}
\]


\noindent
\textbf{Step 4.}
Similarly, the repulsive Ornstein-Uhlenbeck process ${(Z_t)}_{t\geq 0}$ 
\[
\begin{cases}
dZ_t=\sqrt{2\eta_\varepsilon}\,dB_t+\alpha Z_t\,dt,\\
Z_0=-c_\varepsilon/b_\varepsilon,
\end{cases}
\]
evolves, conditionally to the event $\BRA{T_1<T_{-1}}$, as
\[
\begin{cases}
\displaystyle{
dZ_t=\sqrt{2\eta_\varepsilon}dB_t+\PAR{
\alpha Z_t+2\eta_\varepsilon
\frac{g_\varepsilon'(Z_t)}{g_\varepsilon(Z_t)}\ind_\BRA{t\leq T_1}}\,dt
},\\
Z_0=-c_\varepsilon/b_\varepsilon,
\end{cases}
\]
where 
\[
g_\varepsilon(y)=
\frac{\int_{-1}^y\! e^{-\alpha s^2/(2\eta_\varepsilon)}\,ds}
{\int_{-1}^1\! e^{-\alpha s^2/(2\eta_\varepsilon)}\,ds}
\quad\text{and}\quad 
\frac{g_\varepsilon'(y)}{g_\varepsilon(y)}
=\frac{e^{-\alpha y^2/(2\eta_\varepsilon)}}
{\int_{-1}^y\! e^{-\alpha s^2/(2\eta_\varepsilon)}\,ds}.
\]\\

\noindent
\textbf{Step 5.}
Let us now notice that the drifts of the stochastic differential equations 
that drive $Y$ and $Z$ conditionally to the event $\BRA{T_1 < T_{-1}}$ are close. 

\begin{lem}\label{lem:diff-drift}
Under Assumption \ref{as:V}, if $4/9<\beta<1/2$, then
\begin{equation}\label{eq:diff-drift}
\frac{1}{\eta_\varepsilon}\times\sup_{y\in(-1,-1]}\ABS{\PAR{-\frac{V'(b_\varepsilon y)}{b_\varepsilon}
+2\eta_\varepsilon\frac{h_\varepsilon'(y)}{h_\varepsilon(y)}}
-\PAR{\alpha y+2\eta_\varepsilon\frac{g_\varepsilon'(y)}{g_\varepsilon(y)}}}
\xrightarrow[\varepsilon\to 0]{}0.
\end{equation}
\end{lem}

\begin{proof}

Thanks to Assumption \ref{as:V}, as soon as $b_\varepsilon<\delta$, we have, for any $y\in [-1,1]$, 
\[
\ABS{-\frac{V'(b_\varepsilon y)}{b_\varepsilon}-\alpha y}
\leq K b_\varepsilon y^2\leq K b_\varepsilon
\]
so that
$$\frac{1}{\eta_\varepsilon}\times\sup_{y\in(-1,-1]}\ABS{-\frac{V'(b_\varepsilon y)}{b_\varepsilon}
-\alpha y}
\xrightarrow[\varepsilon\to 0]{}0$$
as soon as $\beta>1/3$. It remains to prove that $\sup_{y\in(-1,-1]}|\Delta_\varepsilon(y)|$ goes to zero when $\varepsilon$ goes to zero, where
$$\Delta_\varepsilon(y):=\frac{h_\varepsilon'(y)}{h_\varepsilon(y)}-\frac{g_\varepsilon'(y)}{g_\varepsilon(y)}=\frac{e^{V(b_\varepsilon y)/\varepsilon}}
{\int_{-1}^y\! e^{V(b_\varepsilon s)/\varepsilon}\,ds}-\frac{e^{-\alpha y^2/(2\eta_\varepsilon)}}
{\int_{-1}^y\! e^{-\alpha s^2/(2\eta_\varepsilon)}\,ds}.$$
We propose to do this in two steps: first for $y\in[-1+\varepsilon^\kappa,1]$, then for $y\in(-1,-1+\varepsilon^\kappa]$, where $\kappa=\beta/2$. 
\begin{enumerate}
\item $y\in[-1+\varepsilon^\kappa,1]$: Thanks to assumption \ref{as:V}, we have for all $s\in[-1,1]$
$$V(b_\varepsilon s)=-\alpha b_\varepsilon^2 s^2/2+\theta_\varepsilon(s)b_\varepsilon^3 s^3,$$
with 
$$\sup_{s\in[-1,1]}|\theta_\varepsilon(s)|\leq\frac{1}{6}\sup_{x\in[-b_\varepsilon,b_\varepsilon]}V^{(3)}(x) \le C,$$
where $C$ is a constant independent of $\varepsilon$.
As a consequence, since $b_\varepsilon=\varepsilon^\beta$ and $\eta_\varepsilon=\varepsilon^{1-2\beta}$,
$$e^{V(b_\varepsilon s)/\varepsilon}=e^{-\alpha s^2/(2\eta_\varepsilon)}e^{\theta_\varepsilon(s)s^3\varepsilon^{3\beta-1}}.$$
Now, we can write
$$e^{\theta_\varepsilon(s)s^3\varepsilon^{3\beta-1}}=1+\delta_\varepsilon(s)\theta_\varepsilon(s)s^3\varepsilon^{3\beta-1},$$
with 
$$\sup_{s\in[-1,1]}|\delta_\varepsilon(s)|\leq e^{C\varepsilon^{3\beta-1}}\le \tilde C,$$
where $\tilde C$ is a constant independent of $\varepsilon$.
For the sake of simplicity, we denote $\theta_\varepsilon(s)$ for $\delta_\varepsilon(s)\theta_\varepsilon(s)$. This leads to the following decomposition
$$\frac{h_\varepsilon'(y)}{h_\varepsilon(y)}=\frac{1}{1+\frac{\int_{-1}^y\! e^{-\alpha s^2/(2\eta_\varepsilon)}\theta_\varepsilon(s)s^3\varepsilon^{3\beta-1}\,ds}{\int_{-1}^y\! e^{-\alpha s^2/(2\eta_\varepsilon)}\,ds}}\times\frac{e^{-\alpha y^2/(2\eta_\varepsilon)}}
{\int_{-1}^y\! e^{-\alpha s^2/(2\eta_\varepsilon)}\,ds}\left(1+\theta_\varepsilon(y)y^3\varepsilon^{3\beta-1}\right).$$
Now, let us notice that for any $y\in(-1,1]$,  
$$\ABS{\frac{\int_{-1}^y\! e^{-\alpha s^2/(2\eta_\varepsilon)}\theta_\varepsilon(s)s^3\varepsilon^{3\beta-1}\,ds}{\int_{-1}^y\! e^{-\alpha s^2/(2\eta_\varepsilon)}\,ds}}\leq D\varepsilon^{3\beta-1}$$
with $D$ independent of $\varepsilon$. Consequently
$$\frac{1}{1+\frac{\int_{-1}^y\! e^{-\alpha s^2/(2\eta_\varepsilon)}\theta_\varepsilon(s)s^3\varepsilon^{3\beta-1}\,ds}{\int_{-1}^y\! e^{-\alpha s^2/(2\eta_\varepsilon)}\,ds}}=1-\lambda_\varepsilon(y)\varepsilon^{3\beta-1}$$
and there exists a constant $E$, independent of $\varepsilon$, such that  
$$\sup_{y\in(-1,1]}|\lambda_\varepsilon(y)|<E.$$
Thus we can write
$$\frac{h_\varepsilon'(y)}{h_\varepsilon(y)}=\left(1+\nu_\varepsilon(y)\varepsilon^{3\beta-1}\right)\times\frac{e^{-\alpha y^2/(2\eta_\varepsilon)}}
{\int_{-1}^y\! e^{-\alpha s^2/(2\eta_\varepsilon)}\,ds}$$
and there exists a constant $F$, independent of $\varepsilon$, such that 
$$\sup_{y\in(-1,1]}|\nu_\varepsilon(y)|<F.$$
Finally, for all $y\in(-1,1]$,  we have obtained
\begin{equation}\label{eq:blablabla}
|\Delta_\varepsilon(y)|:=\ABS{\frac{h_\varepsilon'(y)}{h_\varepsilon(y)}-\frac{g_\varepsilon'(y)}{g_\varepsilon(y)}}\leq F\varepsilon^{3\beta-1}\times\frac{e^{-\alpha y^2/(2\eta_\varepsilon)}}
{\int_{-1}^y\! e^{-\alpha s^2/(2\eta_\varepsilon)}\,ds},
\end{equation}
and the goal is now to upper-bound the last term in this equation. In this aim, we first consider the case where $y\in[-1+\varepsilon^\kappa,0]$. In the integral, we make the change of variable
$$v=\varepsilon^{-\gamma}\times\frac{s^2-y^2}{2\eta_\varepsilon}$$ 
with $\gamma=(5\beta-2)/2$, so that $\gamma>0$ as soon as $\beta>2/5$. We get
$$\int_{-1}^y\! e^{-\alpha s^2/(2\eta_\varepsilon)}\,ds=\eta_\varepsilon \varepsilon^\gamma e^{-\alpha y^2/(2\eta_\varepsilon)}I_\varepsilon(y)$$
where
$$I_\varepsilon(y):=\int_0^{\frac{\varepsilon^{-\gamma}(1-y^2)}{(2\eta_\varepsilon)}}\frac{e^{-\alpha \varepsilon^{\gamma}v}}{\sqrt{2\eta_\varepsilon\varepsilon^{\gamma}v+y^2}}\,dv$$ 
Since $\eta_\varepsilon\varepsilon^{\gamma}=\varepsilon^{\beta/2}$ and $y\in[-1+\varepsilon^\kappa,0]$, with $\kappa=\beta/2$, it is clear that for $\varepsilon$ small enough, one has: $\forall y\in[-1+\varepsilon^\kappa,1]$, 
$$I_\varepsilon(y)\geq J_\varepsilon:=\int_0^{1/2}\frac{e^{-\alpha \varepsilon^{\gamma}v}}{\sqrt{2\varepsilon^{\beta/2}v+1}}\,dv\xrightarrow[\varepsilon\to 0]{}\frac{1}{2},$$
so that for $\varepsilon$ small enough, one has $I_\varepsilon(y)\geq 1/4$. Putting all things together gives 
$$|\Delta_\varepsilon(y)|\leq 4F\varepsilon^{3\beta-1}\times\varepsilon^{-\beta/2}=4F\varepsilon^{\gamma}\xrightarrow[\varepsilon\to 0]{}0,$$
and the uniform convergence is proved for $y\in[-1+\varepsilon^\kappa,0]$. In order to conclude for $y\in[-1+\varepsilon^\kappa,1]$, it remains to notice that if $y\in[0,1]$, one has
$$\frac{e^{-\alpha y^2/(2\eta_\varepsilon)}}
{\int_{-1}^y\! e^{-\alpha s^2/(2\eta_\varepsilon)}\,ds}\leq \frac{1}{\int_{-1}^0\! e^{-\alpha s^2/(2\eta_\varepsilon)}\,ds}\underset{\varepsilon\to 0}{\sim}\sqrt{\frac{2 \alpha}{\pi\eta_\varepsilon}}.$$
Coming back to equation (\ref{eq:blablabla}) gives, for all $y\in[0,1]$ and for $\varepsilon$ small enough,
$$|\Delta_\varepsilon(y)|\leq F\varepsilon^{3\beta-1}\times\varepsilon^{\beta-1/2}=F\varepsilon^{4\beta-3/2}\xrightarrow[\varepsilon\to 0]{}0,$$
since $\beta>2/5$. This concludes the case where $y\in[-1+\varepsilon^\kappa,1]$.
\item $y\in(-1,-1+\varepsilon^\kappa]$: let us denote $y=-1+p\varepsilon^\kappa$, with $0<p\leq 1$, so that our goal is now to upper-bound 
$$|\Delta_\varepsilon(p)|:=\ABS{\frac{h_\varepsilon'(-1+p\varepsilon^\kappa)}{h_\varepsilon(-1+p\varepsilon^\kappa)}-\frac{g_\varepsilon'(-1+p\varepsilon^\kappa)}{g_\varepsilon(-1+p\varepsilon^\kappa)}},$$
that is to say
$$|\Delta_\varepsilon(p)|=\ABS{\frac{e^{V(b_\varepsilon (-1+p\varepsilon^\kappa))/\varepsilon}}
{\int_{-1}^{-1+p\varepsilon^\kappa}\! e^{V(b_\varepsilon s)/\varepsilon}\,ds}-\frac{e^{-\alpha (-1+p\varepsilon^\kappa)^2/(2\eta_\varepsilon)}}
{\int_{-1}^{-1+p\varepsilon^\kappa}\! e^{-\alpha s^2/(2\eta_\varepsilon)}\,ds}},$$
independently of $p\in(0,1]$. For any smooth function $f$ on $[-1,0]$, we may write the following Taylor expansions
$$f(-1+p\varepsilon^\kappa)=f(-1)+f'(-1)p\varepsilon^\kappa+\frac{f''(\theta_1)}{2}p^2\varepsilon^{2\kappa}$$
and 
$$\int_{-1}^{-1+p\varepsilon^\kappa}f(s)\,ds=f(-1)p\varepsilon^\kappa+\frac{f'(-1)}{2}p^2\varepsilon^{2\kappa}+\frac{f''(\theta_2)}{6}p^3\varepsilon^{3\kappa}$$
where $\theta_1$ and $\theta_2$ belong to the interval $(-1,-1+\varepsilon^\kappa)$, and depend on $p$ and $\varepsilon$. This leads to 
$$\frac{f(-1+p\varepsilon^\kappa)}{\int_{-1}^{-1+p\varepsilon^\kappa}f(s)\,ds}=\frac{1}{p\varepsilon^\kappa}\times\frac{1+\frac{f'(-1)}{f(-1)}p\varepsilon^\kappa+\frac{1}{2}\frac{f''(\theta_1)}{f(-1)}p^2\varepsilon^{2\kappa}}{1+\frac{1}{2}\frac{f'(-1)}{f(-1)}p\varepsilon^\kappa+\frac{1}{6}\frac{f''(\theta_2)}{f(-1)}p^2\varepsilon^{2\kappa}}.$$
Considering $f(s)=e^{-\alpha s^2/(2\eta_\varepsilon)}$, we get 
$$\frac{f'(-1)}{f(-1)}p\varepsilon^\kappa=\alpha p \varepsilon^{\beta/2}\varepsilon^{1-2\beta}=\alpha p\varepsilon^{\gamma}$$
and
$$\frac{f''(\theta)}{f(-1)}p^2\varepsilon^{2\kappa}=p^2\left(\frac{\alpha^2\theta^2}{\eta_\varepsilon^2}-\frac{\alpha}{\eta_\varepsilon}\right)\varepsilon^{2\kappa}e^{\alpha(1-\theta^2)/(2\eta_\varepsilon)}.$$
Now, since there exists $q\in(0,1)$ such that $\theta=-1+q\varepsilon^\kappa$, it is readily seen that
$$e^{\alpha(1-\theta^2)/(2\eta_\varepsilon)}\underset{\varepsilon\to 0}{\sim}e^{\alpha q\varepsilon^{\gamma}}\xrightarrow[\varepsilon\to 0]{}1$$
and 
$$\frac{f''(\theta)}{f(-1)}p^2\varepsilon^{2\kappa}\underset{\varepsilon\to 0}{\sim}\alpha^2 p^2\varepsilon^{2\gamma}.$$ 
Thus we have the following Taylor expansion
$$\frac{f(-1+p\varepsilon^\kappa)}{\int_{-1}^{-1+p\varepsilon^\kappa}f(s)\,ds}=\frac{1}{p\varepsilon^\kappa}\times\left(1+\frac{\alpha p}{2}\varepsilon^{\gamma}+\phi_\varepsilon(p)p\varepsilon^{2\gamma}\right),$$
with $\sup_{0<p\leq1}|\phi_\varepsilon(p)|<\infty$. 

Now, considering $f(s)=e^{V(b_\varepsilon s)/\varepsilon}$, we get
$$\frac{f'(-1)}{f(-1)}p\varepsilon^\kappa= p \varepsilon^\kappa\times\frac{b_\varepsilon V'(-b_\varepsilon)}{\varepsilon}=\alpha p\varepsilon^{\gamma}+\xi_\varepsilon(p)\varepsilon^{\gamma+\beta}$$
and
$$\frac{f''(\theta)}{f(-1)}p^2\varepsilon^{2\kappa}=p^2\left(\left(\frac{b_\varepsilon V'(b_\varepsilon\theta)}{\varepsilon}\right)^2+\frac{b_\varepsilon^2 V''(b_\varepsilon\theta)}{\varepsilon}\right)\varepsilon^{2\kappa}e^{(V(b_\varepsilon\theta)-V(-b_\varepsilon))/\varepsilon}.$$
For the same reason as above, we have then
$$\frac{f''(\theta)}{f(-1)}p^2\varepsilon^{2\kappa}\underset{\varepsilon\to 0}{\sim}\alpha^2 p^2\varepsilon^{2\gamma}.$$
Since $\beta>\gamma$, we have the following Taylor expansion
$$\frac{f(-1+p\varepsilon^\kappa)}{\int_{-1}^{-1+p\varepsilon^\kappa}f(s)\,ds}=\frac{1}{p\varepsilon^\kappa}\times\left(1+\frac{\alpha p}{2}\varepsilon^{\gamma}+\varphi_\varepsilon(p)p\varepsilon^{2\gamma}\right),$$
with $\sup_{0<p\leq1}|\varphi_\varepsilon(p)|<\infty$. Gathering the intermediate results, we get
$$|\Delta_\varepsilon(p)|=\ABS{\varphi_\varepsilon(p)-\phi_\varepsilon(p)}\varepsilon^{2\gamma-\kappa}\leq G \varepsilon^{(9\beta-4)/2},$$
where $G$ is independent of $\varepsilon$. It turns out that the uniform convergence is ensured as soon as $\beta>4/9$. 
\end{enumerate}
 
 This concludes the proof of Lemma~\ref{lem:diff-drift}.

\end{proof}

\noindent
\textbf{Step 6.}
The difference of the two drifts in \eqref{eq:diff-drift} is negligible with 
respect to the variance $\eta_\varepsilon$ of the Brownian component 
as soon as $4/9<\beta<1/2$. 
Theorem 4.3 in \cite{D92} ensures that $\cL(Y_\cdot\vert T_1<T_{-1})$
and $\cL(Z_\cdot\vert T_1<T_{-1})$ are then asymptotically equivalent. 
This approximation result relies on the Girsanov Theorem. The Novikov 
condition ensuring that the exponential martingale is uniformly integrable can 
be checked as in \cite{D92}. In particular, 
\[
T^Y_1\underset{\varepsilon\to 0}{\sim}T^Z_1.
\]

\noindent
\textbf{Step 7.}
After an obvious scaling, we have to estimate the reactive time for a 
repulsive Ornstein-Uhlenbeck process between $-b_\varepsilon$ 
and $b_\varepsilon$ starting at 
\[
x_\varepsilon=-c_\varepsilon=-\varepsilon^{\gamma}
\xrightarrow[\varepsilon\to 0]{}0.
\] 
Since $b_\varepsilon/\sqrt{\varepsilon}$ and $c_\varepsilon/\sqrt{\varepsilon}$ both go to infinity when $\varepsilon$ goes to zero, the estimates in the proof of Theorem~\ref{th:Gumbel} (see also Remark~\ref{rem:laplace}) ensure that  
\[
T_1^Z \underset{\varepsilon\to 0}{\sim} 
\frac{1}{\alpha}\PAR{-\log \varepsilon+
\log b_\varepsilon+\log c_\varepsilon +G-\log\alpha}
\]
where the law of $G$ is a standard Gumbel distribution. Putting all things together, we have established the following estimate. 

\begin{prop}\label{prop:approx-milieu}
Conditionally to the event $\BRA{T_{b_\varepsilon}<T_A}$,
\[
T_{-c_\varepsilon\to b_\varepsilon} 
\underset{\varepsilon\to 0}{\sim}
\frac{1}{\alpha}
\PAR{-\log \varepsilon+\log c_\varepsilon+\log b_\varepsilon +G-\log \alpha}
\]
where the law of $G$ is a standard Gumbel distribution. 
\end{prop}

\subsection{Conclusion}

The estimates of Propositions \ref{prop:up-down} and 
\ref{prop:approx-milieu} are the key points of the proof of 
Theorem~\ref{th:total}.

\begin{proof}[Proof of Theorem \ref{th:total}]
One can write
\[
-\int_{b_\varepsilon}^B\!\frac{ds}{V'(s)}
=-\frac{1}{\alpha}\left(\int_{b_\varepsilon}^B\!\PAR{\frac{\alpha}{V'(s)}+\frac{1}{s}}\,ds
-\int_{b_\varepsilon}^B\!\frac{ds}{s}\right)
\]
Thanks to Assumption~\eqref{eq:hypV}, $s\mapsto \alpha V'(s)^{-1}+s^{-1}$
is integrable on $(0,B)$. 
\[
t_{b_\varepsilon\to B}=-\frac{\log b_\varepsilon}{\alpha} 
+\frac{\log B}{\alpha}-\frac{1}{\alpha}\int_0^B\!\PAR{\frac{\alpha}{V'(s)}+\frac{1}{s}}\,ds
+o_\varepsilon(1). 
\]
Similarly, 
\[
t_{-c_\varepsilon\to x}=-\frac{\log c_\varepsilon}{\alpha} 
+\frac{\log \vert x\vert }{\alpha}+
\frac{1}{\alpha}\int_x^0\!\PAR{\frac{\alpha}{V'(s)}+\frac{1}{s}}\,ds
+o_\varepsilon(1). 
\]
 As a conclusion, Propositions~\ref{prop:up-down} and \ref{prop:approx-milieu}
ensure that, for any $x\in(A,0)$, we have, conditionally to $\BRA{T_B<T_A}$ that 
\begin{align*}
 T_{x\to B} &=T_{x\to-c_\varepsilon}+T_{-c_\varepsilon\to b_\varepsilon}+
 T_{b_\varepsilon\to B}\\
 &\underset{\varepsilon\to 0}{\sim} 
 t_{-c_\varepsilon\to x}
 +\frac{1}{\alpha}\left(-\log\varepsilon+\log c_\varepsilon +\log b_\varepsilon -\log\alpha+G\right)
 +t_{b_\varepsilon\to B}\\
 &\underset{\varepsilon\to 0}{\sim} 
\frac{1}{\alpha}
\PAR{-\log\varepsilon +\log(\vert x\vert B)+F(x)+F(B)-\log\alpha+G}.
\end{align*}
Notice that one can let $x$ go to $A$ in this expression. 
\end{proof}
Figure \ref{fi:doublepuits} illustrates this result for the process 
$$d X^{(\varepsilon)}_t=
-V'(X^{(\varepsilon)}_t)\,dt+\sqrt{2\varepsilon} dB_t,$$ 
with $V(x)=x^4/4-x^2/2$, $X^{(\varepsilon)}_0=x=-0.89$, on the set $[A,B]=[-0.9,0.9]$. Denoting $T_{-0.89\to 0.9}$ the length of the reactive path from -0.89 to 0.9, then Theorem \ref{th:total} ensures that $\dE[T_{-0.89\to 0.9}]$ is equivalent to $-\log\varepsilon+\log(0.89\times0.9)-\frac{1}{2}\log(1-0.89^2)-\frac{1}{2}\log(1-0.9^2)+\gamma$ (where $\gamma$ stands for the Euler's constant) when $\varepsilon$ goes to zero. Figure \ref{fi:doublepuits} compares this theoretical result (continuous line) with the empirical means obtained thanks to the algorithm described in \cite{CGLP} for $\varepsilon$ ranging from $0.007$ to $1$ (circles).
 
\begin{figure}
\begin{center}
 \includegraphics[scale=0.3]{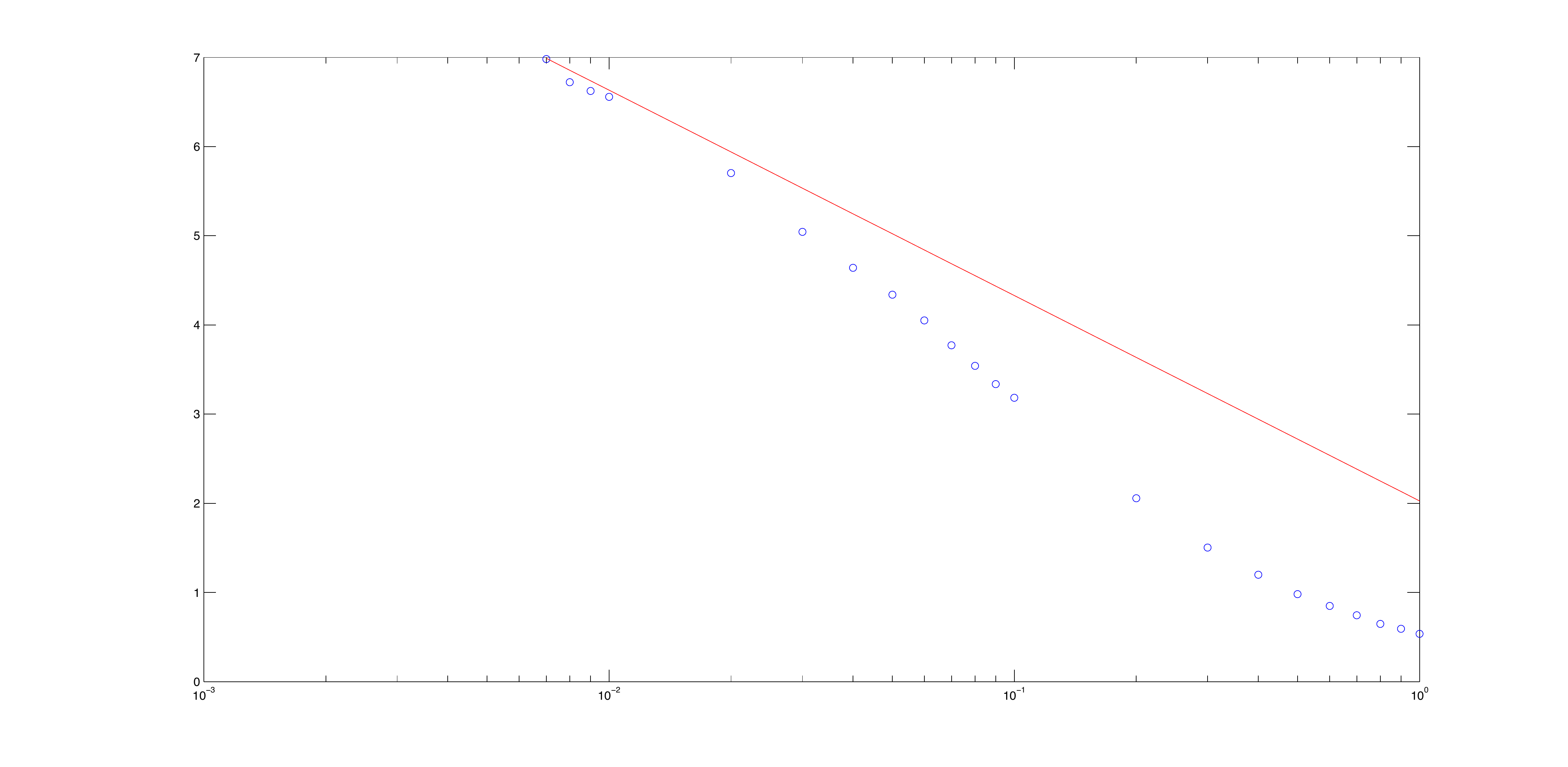}
 \caption{Mean time of the reactive path for the potential $V$ 
 given in \eqref{eq:defV} as a function of $\log\varepsilon$. The 95\% confidence intervals
 are of the size of the points.  These results have been obtained with
 the algorithm described in \cite{CGLP}. The theoretical asymptotic
 behavior (when $\varepsilon$ goes to 0) is drawn in dotted line.}
 \label{fi:doublepuits}
\end{center}
\end{figure}

\section{Other examples}\label{sec:degenerate}

The aim of this section is to analyze the distribution of the lengths of the reactive 
paths, when the potential $V$ has a maximum at point $z^*=0$, but does not 
satisfy the Assumption~\ref{as:V}. More precisely, we successively consider 
three cases:
\begin{enumerate}
\item $V$ behaves like $-|x|$ around $x=0$,
\item $V$ is constant equal to $0$ around $x=0$,
\item $V$ is regular at $0$ but $V''(0)=0$.
\end{enumerate} We will consider special potentials, for which one can derive an 
explicit expression for the asymptotic of the distribution of the lengths the reactive 
paths. We will see that the asymptotic behavior is very different from what we 
obtained in Theorem~\ref{th:total}.

\subsection{Brownian motion with drift}\label{sec:bmwd}

The easiest case to deal with is the one of the singular potential
$V(x)=-\beta |x|$. It corresponds to a Brownian motion with 
a piecewise constant drift, namely:
\[
dX^{(\varepsilon)}_t=
\beta\ \mbox{sgn}\left(X^{(\varepsilon)}_t\right)\,dt+\sqrt{2\varepsilon}dB_t,
\]
where $\beta$ is a positive real number and $\mbox{sgn}(x)$ stands for the sign 
of $x$. In that case, Equation~\eqref{eq:generalu} is a second 
order ordinary differential equation with constant coefficients

Let us recall the expression of the Laplace transform of the conditionned 
first exit time on $(a,b)$ for a Brownian motion with drift (see \cite[p.309]{BS}).
\begin{prop}
Choose $a<x<b$ and $\mu\in\dR$ and 
consider the process $W^{(\mu)}$ defined by $W^{(\mu)}_t=\mu t +W_t$. 
Let us denote by $H$ the first exit time of $(a,b)$. Then, 
\[
\dE_x\PAR{e^{-s H}\vert W^{(\mu)}_H=b}
=\frac{\sinh((b-a)\ABS{\mu})}{\sinh((x-a)\ABS{\mu})}
   \frac{\sinh((x-a)\sqrt{2s +\mu^2})}{\sinh((b-a)\sqrt{2s +\mu^2})}.
\]
\end{prop}
A few remarks are in order.
\begin{rem}
 Notice that the law of $H$ knowing that $T_b<T_a$ does 
 not depend on the sign of the drift $\mu$. This may seem surprising
 but it is consistent with the fact that going up is equivalent to
 going down after introducing the $h$-transformed process, see
 Section~\ref{sec:up} above.
\end{rem}
\begin{rem}
Notice that 
\begin{equation}\label{eq:xx}
\lim_{x\to a}\dE_x\PAR{e^{-s H}\vert W^{(\mu)}_H=b}=
\frac{\sinh((b-a)\ABS{\mu})}{\ABS{\mu}}
   \frac{\sqrt{2s +\mu^2}}{\sinh((b-a)\sqrt{2s +\mu^2})}.
\end{equation}
\end{rem}
\begin{rem}
If $\mu>0$ then $H$ converges to $H_b$ the hitting time of $b$ 
as $a\to -\infty$:  
$$
\lim_{a\to -\infty}\dE_x\PAR{e^{-s H}\vert W^{(\mu)}_H=b}=
e^{\mu(b-x)(1-\sqrt{1+2s/\mu^2})}
$$
which is the Laplace transform of the inverse Gaussian distribution with parameter $m=(b-x)/\mu$ and $l=m^2$. We recall that the density of the inverse Gaussian distribution with parameters $(m,l)$ is $x \mapsto \sqrt{\frac{l}{2\pi}} x^{-3/2}\exp\left(-\frac{l(x-m)^2}{2m^2 x} \right) 1_{x >0}$.
\end{rem}
We can use these results to study the law of the hitting of 0 starting from 
$x=-\delta$ if the process $X^{(\varepsilon)}$ satisfies, at least when $X^{(\varepsilon)}_t \in (-\delta,0)$:
\[
X^{(\varepsilon)}_t=x+\sqrt{2\varepsilon}B_t-\beta t.
\]
From the scaling property of the Brownian motion, we can compute the 
Laplace transform $F$ of $H=\inf\BRA{t\geq 0\ :\ X^{(\varepsilon)}_t\not\in (-\delta, 0)}$ 
conditionally to $\BRA{X^{(\varepsilon)}_H=0}$, using~\eqref{eq:xx}: 
\begin{align*}
F_\varepsilon(s)&=
\frac{\sinh(\delta\beta/(2\varepsilon))}
{\sinh(\delta\sqrt{\beta^2/(2\varepsilon)^2+s/\varepsilon})}
\frac{\sqrt{\beta^2/(2\varepsilon)^2+s/\varepsilon}}{\beta/(2\varepsilon)}
 \\
&=
\frac{\exp\left(\frac{\delta\beta}{2\varepsilon}
\left(1-\sqrt{1+\frac{4\varepsilon s}{\beta^2}}\right) \right)- \exp\left(-\frac{\delta\beta}{2\varepsilon}
\left(1+\sqrt{1+\frac{4\varepsilon s}{\beta^2}}\right) \right)}
{1-\exp\left(-\frac{\delta\beta}{\varepsilon}
\sqrt{1+\frac{4\varepsilon s}{\beta^2}} \right)}
\sqrt{1+\frac{4\varepsilon s}{\beta^2}}.
\end{align*}
For a fixed $s$, we thus get $\lim_{\varepsilon \to 0} F_\varepsilon(s) = \exp\PAR{-\frac{\delta s}{\beta}}$, and
\begin{align*}
\dE&\left(\exp\left(-s\frac{H-\delta/\beta}{\sqrt{\varepsilon}} \right) \right)
=\exp\left(\frac{s\delta/\beta}{\sqrt{\varepsilon}} \right) F_\varepsilon(s/\sqrt{\varepsilon})
\nonumber \\
&=\exp\left(\frac{s\delta}{\beta\sqrt{\varepsilon}} \right) 
\frac{\exp\left(\frac{\delta\beta}{2\varepsilon}
\left(1-\sqrt{1+\frac{4 \sqrt{\varepsilon} s }{\beta^2 }}\right) \right)- \exp\left(-\frac{\delta\beta}{2\varepsilon}
\left(1+\sqrt{1+\frac{4 \sqrt{\varepsilon} s}{\beta^2 }}\right) \right)}
{1-\exp\left(-\frac{\delta\beta}{\varepsilon}
\sqrt{1+\frac{4 \sqrt{\varepsilon}s}{\beta^2 }} \right)}
\sqrt{1+\frac{4 \sqrt{\varepsilon}s}{\beta^2 }}\\
&\underset{\varepsilon\to 0}{\sim} 
\exp\PAR{\frac{s\delta}{\beta\sqrt{\varepsilon}} + \frac{\delta\beta}{2\varepsilon} \left( 1 - 1 -
\frac{2 \sqrt{\varepsilon} s}{\beta^2} + \frac{2 \varepsilon s^2}{\beta^4} \right)}\label{eq:rem-invG}\\ 
&\underset{\varepsilon\to 0}{\sim} 
\exp\PAR{\frac{\delta s^2}{\beta^3}}. \nonumber
\end{align*}
As a consequence, 
\[
H\xrightarrow[\varepsilon\to 0]{a.s.}\frac{\delta}{\beta}
\quad\text{and}\quad
\frac{H-\delta/\beta}{\sqrt{\varepsilon}}
\xrightarrow[\varepsilon\to 0]{\cL}
\cN\PAR{0,\frac{2\delta}{\beta^3}}.
\]

In this case, with the same reasoning as in Section \ref{sec:down}, one can 
deduce that the length of the reactive path between points $-\delta$ and 
$+\delta$ has the deterministic limit $2\delta/\beta$ when $\varepsilon$ tends to 
zero. The absence of any asymptotic randomness in the length of the reactive 
path, in contrast with Theorem \ref{th:total}, is due to the fact
that in this case, we 
do not have $V'(0)=0$. The next situation that we propose to deal with is the 
opposite one, specifically when $V'(x)=0$ in a neighborhood of $0$, and we call it 
the totally flat potential.

\subsection{Totally flat potential}\label{sec:tf}

Let us investigate in this section the case when the potential $V$ is flat 
around the saddle point. More precisely, let us consider the process given by 
$X^{(\varepsilon)}_t=\sqrt{2\varepsilon} B_t$, $b>0$ and 
\[
H=\inf\BRA{t>0,\ X^{(\varepsilon)}_t\notin (-b,b)}.
\]
One has, for any $s\geq 0$, 
\[
F_\varepsilon(s)=\dE_{-b}\PAR{e^{-s H}\vert X^{(\varepsilon)}_H=b}
=\frac{\sqrt{4b^2s/\varepsilon}}{\sinh\PAR{\sqrt{4b^2s/\varepsilon}}}.
\]
Moreover, 
\[
\dE_{-b}\PAR{e^{s H}\vert X^{(\varepsilon)}_H=b}
=\frac{\sqrt{4b^2s/\varepsilon}}{\sin\PAR{\sqrt{4b^2s/\varepsilon}}}
\quad\text{if }0\leq s\leq \frac{\pi^2}{4b^2}\varepsilon. 
\]
Notice that, for any $s\in \left[0, \frac{\pi^2}{4b^2}\varepsilon\right]$, 
\[
\dE_{-b}\PAR{e^{s H}\vert X^{(\varepsilon)}_H=b}
=G\PAR{\frac{4b^2s}{\varepsilon}}
\quad\text{where}\quad 
G(x)=\frac{1}{\sum_{k\geq 0}\frac{(-x)^k}{(2k+1)!}}.
\]
In particular, 
\[
\dE\PAR{H\vert X^{(\varepsilon)}_H=b}=\frac{2b^2}{3\varepsilon}
\quad\text{and}\quad
\dE\PAR{H^2\vert X^{(\varepsilon)}_H=b}=\frac{28b^4}{45\varepsilon^2}
\quad\text{and}\quad
\dV\PAR{H\vert X^{(\varepsilon)}_H=b}=\frac{8b^4}{45\varepsilon^2}.
\]

\begin{lem}\label{lem:brownien}
For any $\varepsilon>0$ and $b>0$, one has, conditionally to 
$X^{(\varepsilon)}_0=x$ and $T_b < T_{-b}$, and in the limit $x \to -b$, 
\[
H^{(\varepsilon)}_{-b,b}=
\frac{b^2}{\varepsilon}\PAR{\frac{2}{3}+\frac{2\sqrt 2}{3\sqrt 5}Y}
\]
where $\dE(Y)=0$, $\dV(Y)=1$ and its Laplace transform is given by 
\[
\dE\PAR{e^{-s Y}}=
\frac{\sqrt{As}}{\sinh\PAR{\sqrt{As}}} e^{Bs}
\quad\text{where}\quad
A=\frac{6\sqrt 5}{\sqrt 2}
\quad\text{and}\quad
 B=\frac{\sqrt 5}{\sqrt 2}.
\]
\end{lem}

\begin{rem}
 Thanks to the scaling property of the Brownian motion, this result is 
 valid for any $\varepsilon>0$.
 \end{rem}

In conclusion, in the case of a totally flat potential, the length of a reactive path goes to infinity at rate $1/\varepsilon$ 
when $\varepsilon$ goes to zero. Again, this is different from the non-degenerate case 
of Theorem~\ref{th:total} where the length of a reactive path goes to
infinity at a slower rate, namely $\log(1/\varepsilon)$.

\subsection{Degenerate concave potentials}

Between the two extreme situations of Section~\ref{sec:bmwd} (where 
$V'(0)\neq 0$) and Section~\ref{sec:tf} (totally flat potential), the main result 
of this paper stated in Theorem~\ref{th:total}  studies the length of a reactive path for a potential $V$ which is non-
degenerate at 0 (also called quadratic case: $V'(0)= 0$ but 
$V''(0)\neq 0$).  In this last section, we briefly discuss 
some intermediate situations, when the second derivative of the potential $V$ is 
equal to 0 at the local maximum 0. Again, we will see that the asymptotic of the length of the reactive path is very 
different from the quadratic case of Theorem~\ref{th:total}. To that end, we focus on monomial 
potentials: the potential $V$ is given by 
\[
V(x)=-\frac{x^{2n+2}}{2n+2}
\quad\text{with}\quad n\geq 1.
\] 
We consider the diffusion process ${(X^{(\varepsilon)}_t)}_{t\geq 0}$ 
solution of 
\begin{equation}\label{eq:Xn}
 X^{(\varepsilon)}_t=x+\sqrt{2\varepsilon}B_t
 +\int_0^t\!(X_s^{(\varepsilon)})^{2n+1}\,ds.
\end{equation}

As will be explained below, in this case, the length of a reactive 
path goes to infinity at rate $\varepsilon^{-\frac{n}{n+1}}$ when $\varepsilon$ 
goes to zero. Notice that when $n$ goes to infinity, 
$\varepsilon^{\frac{-n}{n+1}}$ tends to $1/\varepsilon$, which is
consistent with the scaling obtained in Section~\ref{sec:tf} for a totally flat potential.


For convenience, we drop in the sequel the parameter $\varepsilon$. 
Let us define 
\[
t_\varepsilon=\varepsilon^{-\frac{n}{n+1}},\quad
a_\varepsilon=\varepsilon^{\frac{1}{2n+2}}
\quad\text{and}\quad
b_\varepsilon=\frac{b}{a_\varepsilon},\ 
x_\varepsilon=\frac{x}{a_\varepsilon}
\]
and introduce the process ${(\tilde X_t)}_{t\geq 0}$ defined by 
\[
\tilde X_t=\frac{X_{t_\varepsilon t}}{a_\varepsilon}.
\]
The process ${(\tilde X_t)}_{t\geq 0}$ is solution of the stochastic differential equation 
\begin{equation}\label{eq:tildeX}
\tilde X_t=x_\varepsilon+\sqrt 2 B_t+\int_0^t\! \tilde X_s^{2n+1}\,ds,
\end{equation}
and we have  that 
\[
\BRA{T_b<T_{-b}}=
\BRA{\tilde T_{b_\varepsilon}<\tilde T_{-b_\varepsilon}},
\]
with obvious notation.
On this event, $T_b=t_\varepsilon \tilde T_{b_\varepsilon}$. 
In Equation~\eqref{eq:tildeX}, the parameter $\varepsilon$ only 
appears in the boundary conditions as in Equation~\eqref{eq:v} 
for the Ornstein-Uhlenbeck process. Notice that, in the 
Ornstein-Uhlenbeck case ($n=0$), $t_\varepsilon$ is equal to 1. As in 
the Ornstein-Uhlenbeck case, conditionally to the event $\BRA{\tilde T_{b_\varepsilon}<\tilde T_{-b_\varepsilon}}$, 
$(\tilde X_t)_{t \ge 0}$ is still a Markov process starting from $x_\varepsilon$ 
and solution of 
\[
dY_t=\sqrt{2}\,dB_t+f_\varepsilon(Y_t)\ind_\BRA{\tilde T_{b_\varepsilon}>t}\,dt
\quad\text{with}\quad
f_\varepsilon(y)=-V'(y)+2\frac{e^{V(y)}}{\int_{-b_\varepsilon}^{y}\! e^{V(s)}\,ds}.
\]
We now want to show that $\tilde T_{b_\varepsilon}$, conditionally to
$\BRA{\tilde T_{b_\varepsilon}<\tilde T_{-b_\varepsilon}}$, has a
limit in law when $\varepsilon$ goes to zero. This will show that
$T_b$ (conditionally to the event $\BRA{T_b<T_{-b}}$) scales like
$\varepsilon^{-\frac{n}{n+1}}$, which is the scaling announced above.

The idea is to compare 
${(Y_t)}_{t\geq 0}$ to the solution ${(Z_t)}_{t\geq 0}$ of the following equation 
\begin{equation}\label{eq:comp-infty}
dZ_t=\sqrt{2}dB_t +f(Z_t)\,dt
\quad\text{with}\quad
f(z)=-V'(z)+2\frac{e^{V(z)}}{\int_{-\infty}^{z}\! e^{V(s)}\,ds}. 
\end{equation}
The following lemma ensures that ${(Z_t)}_{t\geq 0}$ goes to $+\infty$ 
in a finite (and integrable) time, even if it "starts from $-\infty$". 
\begin{lem}\label{lemZ}
If ${(Z_t)}_{t\geq 0}$ is solution of Equation~\eqref{eq:comp-infty}
starting from $x \in \dR$, then it goes to $+\infty$ at a (random) finite time
$\tau_e$. Moreover, $\tau_e$ is integrable and it converges almost
surely to an integrable random time when $x$ goes to $-\infty$:
\[
\lim_{x\to-\infty}\dE_x(\tau_e)
=\int_{-\infty}^{+\infty}\!(p(+\infty)-p(y))m(y)\,dy<+\infty,
\]
where
\[
m(x)=\exp\PAR{\int_0^x\!f(z)\,dz}
\quad\text{and}\quad
p(x)=\int_0^x\! \exp\PAR{-\int_0^y\! f(z)\,dz}\,dy=\int_0^x\!\frac{dy}{m(y)}.
\]
\end{lem}

\begin{proof}[Proof of Lemma \ref{lemZ}]
The result on the longtime behaviour of $(Z_t)_{t \ge 0}$ is a consequence of the behavior at infinity of the drift $f$ 
 given by Equation~\eqref{eq:comp-infty}. For any $x < 0$, three 
 successive integrations by parts lead to 
\begin{equation}\label{equi}
\frac{-e^{V(x)}}{x^{2n+1}}\left(1-\frac{2n+1}{x^{2n+2}}\right)
\leq\int_{-\infty}^x\! e^{V(s)}\,ds
\leq \frac{-e^{V(x)}}{x^{2n+1}}\left(1-\frac{2n+1}{x^{2n+2}}+\frac{(2n+1)(4n+3)}
{x^{4n+4}}\right).
\end{equation}
As a by-product, we get that for any $x<-(2n+1)^{\frac{1}{2n+2}}$,
\begin{equation}\label{eq:estim1}
0 < -x^{2n+1}\leq f(x)\leq -x^{2n+1}\left(\frac{2}{1-\frac{2n+1}{x^{2n+2}}}-1\right).
\end{equation}
Let us introduce, for any $n\geq 1$,
\[
C_n=\int_{-\infty}^{+\infty}e^{V(s)}\,ds
=\int_{-\infty}^{+\infty}e^{\frac{-s^{2n+2}}{2n+2}}\,ds.
\] 
For any $x > 0$, we have
\[
\int_{-\infty}^{x}e^{V(s)}\,ds=C_n-\int_{-\infty}^{-x}e^{V(s)}\,ds
\] 
so that the previous computations imply that for any $x>0$
sufficiently large so that
$\frac{e^{V(x)}}{x^{2n+1}}<C_n$, we have
\begin{equation}\label{eq:estim2}
0 < x^{2n+1}\leq f(x)\leq x^{2n+1}+\frac{2e^{V(x)}}{C_n-\frac{e^{V(x)}}{x^{2n+1}}}.
\end{equation}
A quick inspection of the estimates~\eqref{eq:estim1} and~\eqref{eq:estim2} indicates in particular that
 \[
f(x)\underset{\ABS{x}\to+\infty}{\sim} \ABS{x}^{2n+1}. 
\]
As a consequence, the process $({Z_t})_{t\geq 0}$ starting from 
$x\in\dR$ explodes  with probability 1  at a (random) finite time $\tau_e$ and $Z_t\to+\infty$ as 
$t\to \tau_e$ (see for instance~\cite[ch.6]{KS}. 
In short, this is a straightforward consequence of the expression of 
$\dE_x(T_a\wedge T_b)$ that can be found in \cite[ch.6]{KS} and the fact that 
$1/f(x)$ is integrable at $\pm \infty$. Indeed, for any $x\in(a,b)$, 
\[
\dE_x\PAR{T_a\wedge T_b}=
-\int_a^x\!(p(x)-p(y))m(y)\,dy+
\frac{p(x)-p(a)}{p(b)-p(a)}\int_a^b\!(p(b)-p(y))m(y)\,dy
\]
with
\[ 
m(x)=\exp\PAR{\int_0^x\!f(z)\,dz}
\quad\text{and}\quad p(x)=\int_0^x\! \frac{dy}{m(y)}.
\]
One has obviously that $p(b)\to p(+\infty)\in (0,+\infty)$ as $b\to +\infty$, and 
$p(a)\to -\infty$ as $a\to -\infty$. Thus, 
\[
\lim_{\underset{b\to +\infty}{a\to -\infty}}
\frac{p(x)-p(a)}{p(b)-p(a)}\int_a^b\!(p(b)-p(y))m(y)\,dy
=\int_{-\infty}^{+\infty}\!(p(+\infty)-p(y))m(y)\,dy\in (0,+\infty].
\]

Now, to show that $\tau_e$ is integrable (including in the limit $x \to
-\infty$), we need to prove that
\[
\int_{-\infty}^{+\infty}\!(p(+\infty)-p(y))m(y)\,dy<+\infty.
\]
In this aim, let us first notice that for any real number $y$, we have
\begin{align*}
(p(+\infty)-p(y))m(y)&=
\left(\int_y^{+\infty}\! \exp\PAR{-\int_0^x\! f(s)\,ds}\,dx\right)\times 
\exp\PAR{\int_0^y\! f(s)\,ds}\\
&=\int_y^{+\infty}\! \exp\PAR{-\int_y^x\! f(s)\,ds}\,dx.
\end{align*}
Now, from the definition of $f$,  one has clearly $f(s)\geq s^{2n+1}$ for any 
$s \in \dR$. Hence for any $y>0$,
\[
0\leq (p(+\infty)-p(y))m(y)\leq 
e^{\frac{y^{2n+2}}{2n+2}}\int_y^{+\infty}\! e^{\frac{-x^{2n+2}}{2n+2}}\,dx.
\]
The symmetry of the potential $V$ and an integration by parts show that 
for any $y>0$,
\[
\int_y^{+\infty}\! e^{\frac{-x^{2n+2}}{2n+2}}\,dx
=\int_{-\infty}^{-y}\! e^{V(s)}\,ds\leq \frac{e^{\frac{-y^{2n+2}}{2n+2}}}{y^{2n+1}},
\]
so that
\[
 0\leq (p(+\infty)-p(y))m(y) \leq\frac{1}{y^{2n+1}}.
\]
Since $n\geq 1$, the integrability of the function $y\mapsto (p(+\infty)-p(y))m(y)$ 
when $y$ tends to $+\infty$ is established. In order to conclude, we have to 
estimate this quantity when $y$ goes to $-\infty$ as well. For this, let us first 
recall that
\[
f(s)\underset{s\to-\infty}{\sim} \ABS{s}^{2n+1}
\]
so that $p(+\infty)m(y)$ is clearly integrable when $y$ goes to $-\infty$. The 
estimation of the remaining term is slightly more involved. We rewrite it as follows
\[
-p(y)m(y)=
-m(y)\int_{0}^y\! \exp(-F(x))\,dx,
\]
where for any real number $x$, we define $F(x)$ as the primitive of $f$ with 
value 0 at 0 
\[
F(x)=\int_0^x\! f(s)\,ds.
\]
Notice that $\lim_{x \to -\infty} F(x)=-\infty$.
Then an integration by parts gives
\begin{equation}\label{ipp}
\int_{0}^y\! \exp(-F(x))\,dx
=\frac{C_n}{4}-\frac{\exp(-F(y)}{f(y)}-\int_{0}^y\! \frac{f'(x)}{f(x)^2}\exp(-F(x))\,dx.
\end{equation}
Next, we focus on the last term of this equation, namely 
\[
\int_{0}^y\! \frac{f'(x)}{f(x)^2}\exp(-F(x))\,dx.
\]
For this, we first deduce from the definition of $f$ that 
\[
f'(x)=-V''(x)+2\frac{e^{V(x)}}{\int_{-\infty}^{x}\! e^{V(s)}\,ds}\left(V'(x)-
\frac{e^{V(x)}}{\int_{-\infty}^{x}\! e^{V(s)}\,ds}\right).
\]
From Equation (\ref{equi}), we know that
\[
\frac{e^{V(x)}}{\int_{-\infty}^{x}\! e^{V(s)}\,ds}
\underset{x\to-\infty}{\sim}
 |x|^{2n+1},
\]
and more precisely that
\[
V'(x)-\frac{e^{V(x)}}{\int_{-\infty}^{x}\! e^{V(s)}\,ds}
\underset{x\to-\infty}{\sim}
\frac{-(2n+1)}{|x|}.
\]
This leads to
\[
f'(x)\underset{x\to-\infty}{\sim}-(2n+1)x^{2n},
\]
and 
\[
\frac{f'(x)}{f(x)^2}\exp(-F(x))
\underset{x\to-\infty}{\sim}
\frac{-(2n+1)}{x^{2n+2}}\exp(-F(x)).
\]
From this we deduce that 
\[
\int_{-\infty}^0\frac{f'(x)}{f(x)^2}\exp(-F(x))\,dx=-\infty.
\]
Since $\frac{f'(x)}{f(x)^2}\exp(-F(x))=o(\exp(-F(x)))$ when $x$ tends 
to $-\infty$, we have
\[
\int_{0}^y\frac{f'(x)}{f(x)^2}\exp(-F(x))\,dx
\underset{y\to-\infty}{=}
o\left(\int_{0}^y\! \exp(-F(x))\,dx\right)
\]
and coming back to Equation (\ref{ipp}) gives the following asymptotics
\[
\int_{0}^y\! \exp(-F(x))\,dx \underset{y\to-\infty}{\sim} \frac{-\exp(-F(y))}{f(y)},
\]
so that 
\[
-m(y)\int_{0}^y\! \exp(-F(x))\,dx 
\underset{y\to-\infty}{\sim}
\frac{1}{f(y)}\underset{y\to-\infty}{\sim}\frac{1}{|y|^{2n+1}}.
\]
To sum up, we have shown that  
\[
\int_{-\infty}^{+\infty}\!(p(+\infty)-p(y))m(y)\,dy<+\infty.
\]
This ensures that 
\begin{align*}
\dE_x(\tau_e)&=
\lim_{\underset{b\to +\infty}{a\to -\infty}}\dE_x\PAR{T_a\wedge T_b}\\
&=-\int_{-\infty}^x\!(p(x)-p(y))m(y)\,dy+
\int_{-\infty}^{+\infty}\!(p(+\infty)-p(y))m(y)\,dy.
\end{align*}
In particular, $\dE_x(\tau_e)$ is finite for any $x\in\dR$. Finally,
by monotone convergence theorem, $\tau_e$ has a limit almost surely
when $x \to -\infty$ and
\[
\lim_{x\to-\infty}\dE_x(\tau_e)
=\int_{-\infty}^{+\infty}\!(p(+\infty)-p(y))m(y)\,dy<+\infty.
\]
This concludes the proof of Lemma \ref{lemZ}.
\end{proof}

Thanks to Lemma \ref{lemZ}, we see that $T^Z_{a\to b}$ converges
almost surely to a positive and integrable random variable $T^Z_\infty$ as $a\to-\infty$ and $b\to+\infty$. Moreover, 
\[
\dE(T^Z_\infty)
=\int_{-\infty}^{+\infty}\!(p(+\infty)-p(y))m(y)\,dy<+\infty.
\]

Now, notice that the drift $f_\varepsilon$ that drives $Y$ is 
greater than $f$. This ensures 
that if $Z_0=Y_0$ then, almost surely, $Z_t\leq Y_t$, for any $t\in
[0,\tilde T_{b_\varepsilon})$. As 
a consequence, for any $x \in (-b_\varepsilon,b)$, one has 
$T^Y_{x\to b}\leq T^Z_{x\to b}$. By monotone convergence, 
$T^Y_{x_\varepsilon\to b_\varepsilon}$ converges to a random variable which 
is integrable since 
\[
\dE\PAR{T^Y_{x_\varepsilon\to b_\varepsilon}}\leq 
\dE\PAR{T^Z_\infty} < \infty.
\]
To prove this result with full details, one would need to cut reactive
trajectories into pieces, as done in Section~\ref{sec:proof} above for
the quadratic case. This concludes the proof of the fact that $\tilde T_{b_\varepsilon}$, conditionally to
$\BRA{\tilde T_{b_\varepsilon}<\tilde T_{-b_\varepsilon}}$, has a
limit in law when $\varepsilon$ goes to zero, and consequently, that
$T_b$ (conditionally to the event $\BRA{T_b<T_{-b}}$) scales like
$\varepsilon^{-\frac{n}{n+1}}$.\\



\noindent
\textbf{Acknowledgments.} FM thanks the ASPI team of INRIA 
for its hospitality.


\end{document}